\documentclass[12pt,reqno]{amsart}
\usepackage[utf8x]{inputenc}
\usepackage[icelandic, english]{babel}
\usepackage{t1enc}
\usepackage{graphicx}
\usepackage{amsmath,amssymb,amsthm} 
\usepackage{latexsym}
\usepackage{ucs}
\usepackage[intoc]{nomencl}
\usepackage{enumerate,color}
\usepackage{url}
\usepackage[pdfborder={0 0 0}]{hyperref}
\usepackage{appendix}
\usepackage{eso-pic}
\usepackage[sf,normalsize]{subfigure}
\usepackage[format=plain,labelformat=simple,labelsep=colon]{caption}
\usepackage{placeins}
\usepackage{tabularx}

\newcommand{\C}{{\mathbb  C}}
\newcommand{\N}{{\mathbb  N}}

\newcommand{\R}{{\mathbb  R}}

\newcommand{\Z}{{\mathbb  Z}}
\newcommand{\Pro}{{\mathbb  P}}

\newcommand{\Det}{{\operatorname{Det}}}
\usepackage{amsthm}
\usepackage{mathrsfs}

\numberwithin{equation}{section}

\newtheorem{theorem}{Theorem}[section]
\newtheorem{corollary}[theorem]{Corollary}
\newtheorem{lemma}[theorem]{Lemma}
\newtheorem{proposition}[theorem]{Proposition}
\newtheorem{example}[theorem]{Example}

\newtheorem*{remarks*}{Remarks}
\newtheorem*{remark*}{Remark}

\newcommand{\D}{{\mathbb D}}

\renewcommand{\L}{{\mathcal L}}

\newcommand{\scalar}[2]{{\langle#1,#2\rangle}}

\renewcommand{\Im}{{\operatorname{Im}}}

\renewcommand{\Re}{{\operatorname{Re}}}

\makeatletter
\newcommand{\subalign}[1]{%
	\vcenter{%
		\Let@ \restore@math@cr \default@tag
		\baselineskip\fontdimen10 \scriptfont\tw@
		\advance\baselineskip\fontdimen12 \scriptfont\tw@
		\lineskip\thr@@\fontdimen8 \scriptfont\thr@@
		\lineskiplimit\lineskip
		\ialign{\hfil$\m@th\scriptstyle##$&$\m@th\scriptstyle{}##$\crcr
			#1\crcr
		}%
	}
}
\makeatother

\title[Monge-Amp\`ere measures of plurisubharmonic exhaustions]
{Monge-Amp\`ere measures of plurisubharmonic exhaustions 
associated to the Lie norm of holomorphic maps}

\author{Ragnar Sigurdsson and Au\dh unn Sk\'uta Sn\ae bjarnarson}
\date{\today}

\dedicatory{{\small \it Dedicated to the memory of 
Professor J\'ozef Siciak}}

\begin{document}
\baselineskip=17pt	
\maketitle
\begin{abstract}
In this paper we derive formulas for the Monge-Amp\`ere measures 
of functions of the form $\log|\Phi|_c$, where $\Phi$ is a holomorphic
map on a complex manifold $X$ of dimension $n$ with values in 
$\C^{n+1}\setminus\{0\}$ and $|\cdot|_c$ is the Lie norm on $\C^{n+1}$.

\medskip\noindent
\subjclass{{\bf MSC (2010):} {Primary 32U05; Secondary 32Q28, 32U10,  32W20}}
\end{abstract}

\section{Introduction}
\label{sec:1}

\noindent
This paper grew out of a study of a generalization of the
Bernstein-Walsh-Siciak theorem to Stein manifolds by the second
author \cite{Sna:2018}, where the problem of constructing examples of
plurisubharmonic (psh) exhaustion functions which are  maximal 
outside a compact set appears naturally.  
The main result of the paper is the following.

\begin{theorem}\label{th:1.1}
Let $\Phi=(\Phi_0,\dots,\Phi_n)\colon X\to \C^{n+1}\setminus\{0\}$ 
be a holomorphic map on a complex manifold of dimension $n$. 
The function $\log|\Phi|_c$,
where $|\cdot|_c$ denotes the Lie norm  on  $\C^{n+1}$,
is psh on $X$ and a maximal $C^\infty$  function  on
$\Phi^{-1}(\C^{n+1}\setminus \C\R^{n+1})$, 
where $\C\R^{n+1}=\{\alpha a \,;\, \alpha\in \C, a \in \R^{n+1}\}$.
The Monge-Amp\`ere measure 
$\big(dd^c(\log|\Phi|_c)\big)^n$
has no mass on the set 
$$
{\mathcal A}_\Phi=\{z\in X\, ;\, \operatorname{rank} d\Phi(z)<n
\text{ or } \Phi(z) \in  \operatorname{range} d\Phi(z) \}.
$$
If $M= \Phi ^{-1}(\C\R^{n+1})\setminus {\mathcal A}_\Phi$
is non-empty, then it is an $n$-dimensional 
real analytic manifold and 
$\big(dd^c(\log|\Phi|_c)\big)^n$ is the current of integration 
along $M$ of the $n$-form
\begin{equation}
C_n \big(\Phi_0^2+\cdots+\Phi_n^2\big)^{-\frac{(n+1)}2}
\sum_{j=0}^{n}(-1)^{j+1} \Phi_j
d\Phi_0\wedge\cdots\wedge\widehat{d\Phi_j}\wedge\cdots\wedge
d\Phi_n
\label{eq:1.1}  
\end{equation}
where $C_n=(-1)^{\frac{n(n-1)}{2}} n!\Omega_n$,
$\Omega_n$ is 
 the volume of the
unit-ball in $\R^n$,
and the holomorphic  square root
of the function $\Phi_0^2+\cdots+\Phi_n^2$
is chosen with the same argument as the vector $\Phi$ in $\C\R^{n+1}$.
\end{theorem}

The Monge-Amp\`ere measure of a $C^2$ psh function $v$ on 
an open subset of $\C^n$ is defined as
\begin{align*}
(dd^c v)^n=\underbrace{(dd^c v)\wedge\cdots\wedge (dd^c v)}_{n \text{
  times}}
=4^nn!\Det\big(\tfrac{\partial^2 v}{\partial z_j\partial \bar{z}_k}\big)_{j,k}dV,
\end{align*}
where $d^c=i(\bar{\partial}-\partial)$ and
$dV=\bigwedge_{j=1}^{n}\tfrac i2  dz_j\wedge d\bar{z}_j$ is the
standard Euclidean volume form on $\C^{n}$. By the fundamental work of
Bedford and Taylor \cite{Bed:1976,Bed:1982} one can define the
Monge-Amp\`ere measure of a  $L^\infty_{\text{loc}}$ psh function $v$
as 
\begin{align*}
(dd^c v)^n:=\lim_{k\to\infty}(dd^c v_k)^n
\end{align*}
where the limit is taken in the weak$^*$ topology of measures and $(v_k)$ is any sequence of smooth psh functions decreasing to $v$.

The Lie norm on $\C^N$  is the largest norm on $\C^N$ 
which extends the Euclidean norm $|\cdot|$ on $\R^N$ 
to a {\it complex} norm on $\C^N$. It thus dominates the 
standard Hermitian norm on $\C^N$.  It is given by an explicit
formula in Proposition \ref{prop:2.1},  which shows that it is 
a $C^\infty$ function on $\C^N\setminus \C\R^N$. 
By Bos, Levenberg, Ma'u, and Piazzon
\cite{Bos:2017}  the $0$ eigenspace of the Levi matrix
of $\log|\cdot|_c$ on $\C^N$ at every $\zeta\neq 0$ is spanned by 
$\zeta$ and $\bar \zeta$. This is the case $\varepsilon=0$ in
Theorem~6.1 below.
These vectors are linearly independent if and only if $\zeta$ is 
in $\C^N\setminus \C\R^N$, so the kernel of the Levi matrix of 
the function $\log|\cdot|_c$  
is two dimensional at every $\zeta\in \C^N\setminus\C\R^N$.
Hence  $\log|\Phi|_c$ is maximal on $\Phi^{-1}(\C^{n+1}\setminus \C\R^{n+1})$ 
for every holomorphic map $\Phi\colon X\to \C^{n+1}$ on a
manifold of dimension $n$. This is  the advantage of the Lie norm 
rather than, say, the usual Euclidean
norm for constructing maximal psh functions.
For a general $\Phi\colon X\to \C^{n+1}$ the function 
$\log|\Phi|$ is not maximal on any open subset of $X$.

Recall that a Stein manifold $X$
is said to be {\it parabolic} if every bounded psh function on 
$X$ is constant, it is said to be $S$-{\it parabolic} if there 
exists a psh exhaustion on $X$ which is maximal outside a compact
subset of $X$.  Such an exhaustion is called a {\it parabolic
  potential} or a {\it special psh exhaustion function}.  
The Stein manifold  is said to be $S^*$-{\it parabolic}
if there exists a continuous parabolic potential on $X$.   
See Aytuna and Sadullaev \cite{Ayt:2011,Ayt:2014,Ayt:2016}
for a study of these classes of manifolds.

In order to understand when we can expect to find a 
parabolic potential of the form $\log|\Phi|_c$, let us
look at a slightly more general problem.  Let 
$\Phi\colon X\to Y$ be a holomorphic map on a manifold 
$X$ of dimension $n$ into a manifold $Y$ of
dimension $N$ and let $h$ be a psh function on $Y$.  If $h$ is an exhaustion
and $\Phi$ is a proper map, then $h\circ \Phi$ is
an exhaustion on $X$.  If $h$ is a $C^2$ function 
on an open subset $V$ of $Y$
then $h$ is maximal on $V$ if and only if the
Levi matrix of $h$ has eigenvalue $0$ at every 
point $\zeta\in V$. Recall that 
the {\it Levi form} of $h$ is the 
Hermitian form 
$w\mapsto \L_h(\zeta,w)=\sum_{j,k=1}^N h_{j\bar k}(\zeta)w_j\bar w_k$,
where the Levi matrix $(h_{j\bar k}(\zeta))_{j,k=1}^N$, is given in local coordinates 
$\zeta=(\zeta_1,\dots,\zeta_N)$ as 
$h_{j\bar k}=\partial^2h/\partial\zeta_j\partial \bar \zeta_k$
and  $(w_1,\dots,w_N)$ are the coordinates of the tangent vector 
$w$  at $\zeta$ with respect to the
basis $(\partial/\partial \zeta_1,\dots,\partial/\partial \zeta_N)$.
We have
\begin{equation*}
  \label{eq:1.3}
  \L_{h\circ\Phi}(z;w)=\L_h(\Phi(z),d\Phi(z)(w)), \qquad 
z\in \Phi^{-1}(V), \  w\in T_zX. 
\end{equation*}
From this equation
we see that if for every $z\in \Phi^{-1}(V)$ there exists a tangent vector 
$w\neq 0$ of $X$ at $z$ which is mapped by $d\Phi(z)$ to the kernel of
the Levi matrix of $h$ at $\Phi(z)$, then $h\circ\Phi$ is maximal on 
$\Phi^{-1}(V)$.
If at every $\zeta\in V$ the kernel of the Levi matrix of $h$ at $\zeta$ 
has  dimension at least $k\geq 1$ and $N=n+k-1$, then it follows
by the rank-nullity theorem from linear algebra that
$h\circ \Phi$ is maximal on $\Phi^{-1}(V)$.

Our observations can now be summarized as follows:

\begin{corollary}
If $X$ is a Stein manifold of dimension $n$,
and there  exists a proper holomorphic map  
$\Phi\colon X\to \C^{n+1}$,  
such that $\Phi^{-1}(\C\R^{n+1})$ is compact,
then $X$ is $S^*$-parabolic.   More generally, if 
there  exists a proper holomorphic map  $\Phi\colon X\to \C^N$, 
$N\geq n$,
such that  $\Phi^{-1}(\C\R^N)$ is compact and such
that  for every $z\in X$ there exists a tangent vector $w\neq 0$
at $z$ which  is mapped by the derivative 
$d\Phi(z)$ at $z$  into the span of 
$\Phi(z)$ and $\overline{\Phi(z)}$, then $X$ is
$S^*$-parabolic.   
\end{corollary}

There is an extensive literature  
on the existence of proper holomorphic maps on Stein manifolds, see
e.g., Forstneri{\v c} \cite{For:2017}.

The plan of the paper is as follows. In Section \ref{sec:2}
we review the general theory of the cross norm 
$\|\cdot\|_c$ on 
the complexification $V_\C$ of a real normed 
space $(V,\|\cdot\|)$.
For an inner product space 
$(V,\scalar \cdot\cdot)$ with norm
$|\cdot|$ we  derive a 
simple formula for the cross norm, $|\zeta|_c=|a|+|b|$, 
where the vector $\zeta\in V_\C$ is represented as 
$\zeta=e^{i\theta}(a+ib)$,  with $\theta\in \R$, $a,b\in V$,
$\scalar ab=0$, and $|b|\leq |a|$.
From this formula we
give a new proof of the explicit formulas for $|\zeta|_c$ originally
proved by Dru\.zkowski in \cite{Dru:1974}.     
If $|\cdot|$ is the Euclidean norm on $\R^N$, then
$|\cdot|_c$ is called the Lie norm on $\C^N$.

In Section \ref{sec:3} we first take a look at
the set $\C\R^N$ from a few different viewpoints.  
Then we look at
the representation of $\zeta\in \C^N$ as
$\zeta=e^{i\theta}(a+ib)$ with $\theta\in \R$,
$a,b\in \R^N$, 
$\scalar ab=0$, and $|b|\leq |a|$, which
is very useful in our calculations.
We derive a formula for the pullback of the 
volume form $dV$ on $\C^N$ to the $2N$-dimensional 
real manifold 
$\{(\theta,a,b)\,;\, \scalar ab=0, |b|< |a|\}$.
This formula is applied in Section \ref{sec:7} for calculating 
weak limits of Monge-Amp\`ere measures.

In Sections \ref{sec:4}  and \ref{sec:5} we apply Theorem \ref{th:1.1} to calculate the Siciak-Zakharyuta extremal function and the corresponding equilibrium measure for a few compact subsets $K$ of $\R^n\subset \C^n$. The extremal function has been extensively studied
for several decades and, in particular, authors have shown interest in
the case when $K$ is a compact subset of $\R^N$
\cite{Bar:1992,Bed:1986,Blo:2009,Bur:2005,Bur:2010,Bur:2014,Lun:1985}. Among
our examples is a result that motivated this study \cite{Bos:2017}.

As a preparation for the proof of Theorem \ref{th:1.1}
we analyze in Section \ref{sec:6} the Levi form of 
the maximal function  $h_\varepsilon=\log v_\varepsilon$, where the 
family of functions, 
$v_\varepsilon(\zeta)=(|a|+\varepsilon|b|)^{\frac 12}
+(|b|+\varepsilon|a|)^{\frac 12}$, $\varepsilon\geq 0$,
regularizes 
the Lie norm $v(\zeta)=|\zeta|_c=|a|+|b|$.
Finally, in Section \ref{sec:7} we complete the proof of
Theorem \ref{th:1.1} by  calculating the Monge-Amp\`ere 
measure $\big(dd^c(\log|\Phi|_c)\big)^n$.
as the weak limit of $\big(dd^c(\log v_\varepsilon\circ\Phi)\big)^n$

\subsubsection*{\bf Acknowledgment} This work was partially 
supported by 
by The Icelandic Centre for Research (Rannís), grant no.~152572-052,
and a doctoral grant from The University of Iceland Research Fund.
The authors would like to thank the anonymous referee for comments and
suggestions that led to considerable improvements of the paper.

\section{Lie norm}
\label{sec:2}

\noindent
The Lie norm on $\C^N$ is a special case of an extension of a norm on
a real vector space $V$ to a complex norm on its complexification
$V_{\C}=\C\otimes_\R V$, i.e., the extension satisfies
$\|\alpha \zeta\|=|\alpha|\|\zeta\|$ for 
$\alpha\in \C$ and $\zeta\in V_\C$. There are a few different ways of introducing the complexification,
see Munoz, Sarantopoulos, and Tonge \cite{Mun:1991}. 
One is to define  $V_\C$ as  $V\times V$ with usual
addition and complex multiplication $(\alpha+i\beta)(\xi,\eta)=
(\alpha \xi-\beta \eta,\beta \xi+\alpha \eta)$
where $\alpha,\beta\in \R$ and $\xi,\eta\in V$.  
We have  $(0,\eta)=i(\eta,0)$, so if we identify
the real vector $\xi$ in $V$ with $(\xi,0)$ in $V_\C$, then
we can write  every vector $\zeta=(\xi,\eta)$ as $\zeta=\xi+i\eta$.  
Thus $V$ is a real subspace
of $V_\C$ and $V_\C=V \oplus iV$.  The vectors $\xi$ and $\eta$ are called 
the {\it real} and {\it imaginary} parts of $\zeta$, denoted
by $\Re \zeta$ and $\Im \zeta$.

In \cite{Mun:1991} it is shown that 
every norm $\|\cdot\|$ on $V$  can be extended to a complex
norm on $V_\C$.  
If we have such an extension and $\zeta=\sum_{j=1}^k\alpha_j
\xi_j$ is a vector in $V_\C$, where $\alpha_j\in \C$ and
$\xi_j\in V$,    
then the triangle inequality implies
$\|\zeta\|\leq \sum_{j=1}^k |\alpha_j|\|\xi_j\|$.  
For $\zeta\in V_\C$ we define
\begin{equation}
  \label{eq:2.1}
  \|\zeta\|_c=
\inf\Big\{\sum_{j=1}^k |\alpha_j|\|\xi_j\| \, ;\, 
\zeta= \sum_{j=1}^k \alpha_j\xi_j, \,  \alpha_j\in \C, \,
\xi_j\in V, k\in \N^*\Big\}.
\end{equation}
Since  we can extend every norm on $V$ to a complex
norm on $V_\C$,  it is easy to see 
that $\zeta\mapsto \|\zeta\|_c$ is a norm and moreover it is the largest 
complex norm on $V_\C$ which extends $\|\cdot\|$ from  $V$.
The norm $\|\cdot\|_c$ on $V_\C$ is called the crossnorm
of $\|\cdot\|$.  If  $|\cdot|$ denotes  the Euclidean norm on 
$\R^N$ then the cross norm $|\cdot|_c$ is called 
the Lie norm on $\C^N$.

The set 
$\C V=\{ \alpha \xi   \in V_\C \, ;\, \alpha\in \C, \xi\in V\}
=\{ e^{i\theta}  \xi   \in V_\C \, ;\, \theta \in \R, \xi\in V\}$
consisting  of all vectors in $V_\C$ with parallel  real and imaginary
parts is important in our calculations.  If 
$\|\cdot\|$ extends from $V$ to a complex norm on 
$V_\C$, also denoted by $\|\cdot\|$, 
then we denote the distance 
from $\zeta\in V_\C$ to $\C V$ by
$d(\zeta,\C V)=\inf\{\|\zeta-e^{i\theta}\xi\| \,;\, \theta\in \R,
\xi\in V\}$. For $\zeta=e^{i\theta}\xi\in
\C V$ we have $\|\zeta\|=\|\xi\|$.

Assume from now on that $V$ is an inner product space with norm 
$\xi\mapsto |\xi|=\scalar \xi\xi^{\frac 12}$, where 
$(x,\xi)\mapsto \scalar x\xi$ is an inner product on 
$V$.   The bilinear form on $V$ has a unique extension to a 
symmetric $\C$-bilinear form  on $V_\C$ by the formula
\begin{equation}
  \label{eq:2.2}
  \scalar{x+iy}{\xi+i\eta}=\scalar x\xi-\scalar y\eta
+i\scalar x\eta+i\scalar y\xi,
\end{equation}
and the Hermitian form $(z,\zeta)\mapsto \scalar z{\bar\zeta}$
generates an extension of the norm  by the formula
\begin{equation}
  \label{eq:2.3}
  |\zeta|=\scalar{\zeta}{\bar \zeta}^{\frac 12}
=\big(|\xi|^2+|\eta|^2\big)^{\frac 12}, \qquad \zeta=\xi+i\eta \in V_\C.
\end{equation}
Take $\zeta\in V_\C$ and  $\theta\in \R$ such that 
$e^{-2i\theta}\scalar \zeta\zeta$ is a positive real number.
If $a=\Re(e^{-i\theta}\zeta)$ and $b=\Im(e^{-i\theta}\zeta)$,
then  
$e^{-2i\theta}\scalar \zeta\zeta=|a|^2-|b|^2+2i\scalar ab \geq 0$,
so we conclude that $\scalar ab=0$ and $|b|\leq|a|$. 

\begin{proposition} \label{prop:2.1}
 If $\zeta=\xi+i\eta=e^{i\theta}(a+ib)$, 
$\theta\in \R$,  $\xi,\eta,a,b\in V$, 
$\scalar ab=0$, and $|b|\leq |a|$, then 
\begin{equation}
  \label{eq:2.4}
  |a|=\tfrac 1{\sqrt 2}\big(|\zeta|^2+|\scalar \zeta\zeta|\big)^{\frac 12}
\quad \text{ and } \quad
  |b|=\tfrac 1{\sqrt 2}\big(|\zeta|^2-|\scalar \zeta\zeta|\big)^{\frac 12}.
\end{equation}
The $|\cdot|$-distance from $\zeta$ to $\C V$ 
is $d(\zeta,\C V)=|b|$ and the cross norm is given by the formula
\begin{align}
  \label{eq:2.6}
|\zeta|_c&=|a|+|b|
= \big(|\zeta|^2-d(\zeta,\C V)^2\big)^{\frac 12}+d(\zeta,\C V)\\
&= \big(|\zeta|^2 +\big(
|\zeta|^4-|\scalar \zeta\zeta|^2\big)^{\frac 12}\big)^{\frac 12}
\nonumber
\\
&=\big(|\zeta|^2 +2\big(|\xi|^2|\eta|^2-\scalar \xi\eta^2\big)^{\frac
  12}\big)^{\frac 12}.
\nonumber
\end{align}
If $V$ is a Hilbert space and  
$\scalar \zeta\zeta\neq 0$, then $e^{i\theta}a$ is a unique 
point in $\C V$ with minimal distance to $\zeta$.
If $\scalar \zeta\zeta=0$, then 
$\{e^{i\theta}\Re(e^{-i\theta}\zeta) \,;\, \theta \in \R\}
=\{\tfrac 12 \big(\zeta +e^{2i\theta}\bar \zeta\big) \,;\, 
\theta \in \R\}$
is a circle
consisting of all points in $\C V$ with minimal distance to $\zeta$.
\end{proposition}

\begin{proof}   
Equations (\ref{eq:2.4}) follow from the fact that
$|\zeta|^2=|a|^2+|b|^2$, and $|\scalar \zeta\zeta|=|a|^2-|b|^2$.
For proving the first equality
in (\ref{eq:2.6}) we observe that $\zeta=e^{i\theta}a+e^{i\theta}ib$
  is one of the linear combinations in (\ref{eq:2.1}), so we have
$|\zeta|_c\leq |a|+|b|$ and equality holds if $b=0$.  Assume 
that $b\neq 0$ and take $e_a=a/|a|$ and $e_b=b/|b|$.
 If $\zeta= \sum_{j=1}^k \alpha_j\xi_j$ is some other 
representation of $\zeta$
as in (\ref{eq:2.1}), then
$a=\sum_{j=1}^k \Re\big(e^{-i\theta}\alpha_j\big)\xi_j$, 
 $b=\sum_{j=1}^k \Im\big(e^{-i\theta}\alpha_j\big)\xi_j$, so
the Cauchy-Schwarz inequality and  Pythagoras' theorem give  
\begin{multline*}
  |a|+|b| = \scalar a{e_a}+\scalar b{e_b}
=\sum_{j=1}^k\scalar{ \Re\big(e^{-i\theta}\alpha_j\big)e_a+
\Im\big(e^{-i\theta}\alpha_j\big)e_b}{\xi_j} \\
\leq
\sum_{j=1}^k\big|
\Re\big(e^{-i\theta}\alpha_j\big)e_a+
\Im\big(e^{-i\theta}\alpha_j\big)e_b
\big|\big|\xi_j\big|
=\sum_{j=1}^k |\alpha_j||\xi_j|.
\end{multline*}
This proves the first equality in (\ref{eq:2.6}) and the others
follow from (\ref{eq:2.4}).  
The proof of the remaining statements is identical to  the proof of
Lemma 5.18 in \cite{Hor:1998}.  
\end{proof}

Observe that the last two 
expressions for the Lie norm in (\ref{eq:2.6})  
were first proved by Dru\.zkowski
\cite{Dru:1974}.

\section{Geometry of $\C\R^N$}

\label{sec:3}

\noindent
From now on we let  $\scalar \cdot\cdot$ denote the natural bilinear form
on $\C^N$, $\scalar z\zeta =\sum_{j=1}^N z_j\zeta_j$, which means that
the Hermitian form is $(z,\zeta)\mapsto \scalar z{\bar \zeta}$.    
We let $|\cdot|$ denote the Hermitian norm on $\C^N$ and
$|\cdot|_c$ denote the Lie norm.

The set $\C\R^N$ consists of all vectors in $\C^N$ 
with parallel real and imaginary parts.    Since
$\C\R^N\setminus\{0\}$ is the inverse image of the real projective
space 
$\Pro^{N-1}(\R)$  under the natural map from
$\C^N\setminus\{0\}$ it is a real analytic manifold of dimension
$N+1$.  Moreover, the 
parametrization 
\begin{equation}
  \label{eq:3.1}
  \R\times (\R^N\setminus\{0\}) \ni(\theta,a)
\mapsto e^{i\theta}a \in \C\R^N\setminus\{0\}
\end{equation}
has an injective differential at every point  
and  $\C\R^N\setminus \{0\}$ 
is a fiber bundle over ${\mathbb S}^1$
with the projection 
\begin{equation}
  \label{eq:3.2}
  \C\R^N\setminus \{0\} \ni \zeta=e^{i\theta}a\mapsto 
\dfrac{\scalar \zeta\zeta}{|\zeta|^2}=e^{2i\theta}\in {\mathbb S}^1
\end{equation}
and fiber $\R^N\setminus\{0\}$.   
We have
$|\zeta|^2-|\scalar\zeta\zeta|=4(|\xi|^2|\eta|^2-\scalar \xi\eta^2)$
for every  $\zeta=\xi+i\eta\in \C^N$, 
which shows that $\C\R^N$ is a real algebraic variety,
\begin{equation}
  \label{eq:3.3}
  \C\R^N=\{\xi+i\eta \in \C^N \, ;\,  |\xi||\eta|=|\scalar \xi\eta|\}
=\{\zeta\in \C^N \, ;\,  |\zeta|^2=|\scalar\zeta\zeta|\}.
\end{equation}
If $\zeta\in \C\R^N$ is  written as $e^{i\theta}a$, then 
the argument $\theta$ is determined up to a multiple of $\pi$ and
the vector $a$ up to a sign.

In Section 2 we have seen how useful the $(\theta,a,b)$
coordinates are for the calculation of the Lie norm.  
We are going to use these coordinates to parameterize  
the set $\{\zeta\in \C^N \,;\, \scalar \zeta\zeta\neq 0\}$
and to express the standard volume form on $\C^N$,
\begin{equation}
   \label{eq:3.5}
dV=\big(\tfrac{i}{2}\big)^{N}
\bigwedge_{j=1}^{N}dz_j\wedge d\bar{z}_j
\end{equation}
with respect to these coordinates.

\begin{proposition}
\label{prop:3.1}
The pullback of the volume form $dV$ 
on $\C^N$, under the map 
$(\theta,a,b)\mapsto e^{i\theta}(a+ib)$ 
to the $2N$ dimensional real manifold 
\begin{equation}
  \label{eq:3.6}
  L=\{(\theta,a,b) \,;\, \theta \in \R, a,b\in \R^N, 
\scalar ab=0, |b|<|a|\} \subset \R^{2N+1}
\end{equation}
is given by the formula
\begin{equation}
\label{eq:3.7}
dV=-d\theta\wedge
  \sum_{j=1}^{N}\left(a_j\Lambda_{b_j}+b_j\Lambda_{a_j}\right) 
\end{equation}
where the $2N-1$ forms $\Lambda_{a_j}$ and $\Lambda_{b_j}$
are given by
\begin{align}
\Lambda_{a_j}&=da_1\wedge db_1\wedge\cdots\wedge
  \widehat{da_j}\wedge db_j\wedge\cdots\wedge da_{N}\wedge db_{N}
\label{eq:3.8}\\
\Lambda_{b_j}&=da_1\wedge db_1\wedge\cdots\wedge
  da_j\wedge \widehat{db_j}\wedge\cdots\wedge da_{N}\wedge db_{N},
\label{eq:3.9} 
\end{align}
and $ \widehat{da_j}$ and $\widehat{db_j}$ is the standard notation
for omitted factors. On the open subset  
$L_k:=\{(\theta,a,b)\in L \, ;\,  a_k\not=0 \}$ of $L$ 
we can express the volume form as 
\begin{equation}
\label{eq:3.10}
	dV=\frac{|b|^2}{a_{k}}d\theta\wedge
  \Lambda_{b_{k}}-d\theta\wedge
  \sum_{j=1}^{N}a_j\Lambda_{b_j}.
\end{equation}
\end{proposition}

\begin{proof}
We have $z_j=e^{i\theta}(a_j+ib_j)$ and
$\bar{z}_j=e^{-i\theta}(a_j-ib_j)$, so
\begin{align*}
dz_j&=ie^{i\theta}\big((a_j+ib_j)d\theta-ida_j+db_j\big) 
\quad \text {and} \quad  \\
d\bar{z}_j&=-ie^{-i\theta}\big((a_j-ib_j)d\theta+ida_j+db_j\big)
\end{align*}
and we get
\begin{align}
dz_j\wedge d\bar{z}_j
=-2i(-a_jd\theta\wedge da_j-b_jd\theta\wedge db_j+da_j\wedge db_j).\label{eq:3.10a}
\end{align}
Now equation (\ref{eq:3.7}) 
follows by wedging equation (\ref{eq:3.10a}) over all $j\in\{1,\dots,N\}$, and noticing that 
$\bigwedge_{j=1}^{N} da_j\wedge db_j=0$, which follows
from the equation $a_1b_1+\cdots a_{N}b_{N}=0$.
Furthermore, this equation implies that 
$b_j\Lambda_{a_j}=-\frac{b_{j}^2}{a_{k}}\Lambda_{b_{k}}$
on  $L_{k}$  for every $j$ and $k$. 
If we combine this with (\ref{eq:3.7}), then (\ref{eq:3.10}) 
follows.
\end{proof}

In some cases the variables 
$\theta,a_1,\dots,a_{N},b_1,\dots,b_{N}$ on $L$
are  inconvenient to work with  because they are 
dependent through the equality $\langle a,b\rangle=0$. Therefore we
define the set 
\begin{align*}
\tilde{L}=\{(\theta,a,\beta)\in  \R\times \R^{N}\times
  \R^{N-1} \, ;\,  |\beta|<|a|\} 
\end{align*}
and present the following proposition.

\begin{proposition}\label{prop:3.2}
For each $m\in \{1,\dots,N\}$ there exists a change of
variables  $L_m:=\{(\theta,a,b)\in L \,;\, a_m\neq 0 \}\to \tilde{L}$,
$(\theta,a,b)\mapsto (\theta,a,\beta)$
such that $|\beta|=|b|$ and
\begin{gather*}
dV=-(-1)^{\frac{N(N+1)}{2}}
\big(|a|-\tfrac{|\beta|^2}{|a|}\big)d\theta\wedge dV_a
\wedge dV_{\beta}, 
\quad \text{ where } \\
dV_a=da_1\wedge\cdots\wedge da_{N}
\qquad \text{and}\qquad 
dV_{\beta}=d\beta_1\wedge\cdots\wedge d\beta_{N-1}.
\end{gather*}
\end{proposition}

\begin{proof}
For each $a$ with $a_m\neq 0$ we can define
an orthogonal matrix $U_a=(u_{j,k}(a))$ whose last column is parallel
to $a$ and  the entries of the matrix $U_a$ 
are smooth as functions of $a$. 
This can be done by applying the Gram-Schmidt method on the 
basis 
$\{a,e_1,\dots,e_N\}\setminus \{e_m\}$,
where $\{e_1,\dots,e_N\}$ is the standard basis.
Now we define the change of variables 
$L_m\to \tilde{L}$, $(\theta,a,b)\mapsto (\theta,a,\beta)$
by 
\begin{align*}
  (\beta_1,\dots,\beta_{N-1},0)^t=U_a^{-1}(b_1,\dots,b_{N})^t.
\end{align*}
Since $U_a$ is orthogonal its cofactors are given by the
equation $C_{j,k}=u_{j,k}$, in particular we have
$C_{j,N}=a_{j}/|a|$. 
Therefore, for any $s\in \{1,\dots,N \}$, we have
\begin{align}
(-1)^{\frac{N(N-1)}{2}}&\Lambda_{b_{s}}
=(-1)^{s+N}da_1\wedge\dots\wedge da_{N}\wedge
db_1\wedge\dots\wedge\widehat{db_s}
\wedge\dots\wedge db_{N} \nonumber\\ 
	&=(-1)^{s+N}dV_a \bigwedge_{j=1,j\not=s}^{N}\left(\sum_{k=1}^{N-1}(u_{j,k}d\beta_k+\beta_kdu_{j,k})\right)\nonumber\\
	&=(-1)^{s+N}dV_a\bigwedge_{j=1,j\not=s}^{N}\left(\sum_{k=1}^{N-1}u_{j,k}d\beta_k\right)=C_{s,N}dV_a\wedge dV_{\beta}\nonumber\\
	&=\frac{a_s}{|a|}dV_a\wedge dV_{\beta}. \label{eq:3.11}
\end{align}
The result follows by combining equation 
(\ref{eq:3.11}) with equation (\ref{eq:3.10}).
\end{proof}

\section{Applications}
\label{sec:4}

\noindent
In this section we apply Theorem \ref{th:1.1} to give explicit
formulas for the Siciak-Zakharyuta extremal functions and the
corresponding equilibrium measures for a few examples of compact
sets. We have grouped the examples into three classes; namely we
consider compact subsets of the Euclidean space $\C^n$, the projective
space $\Pro^n(\C)$, and the complex torus $\C^n/\Z^n$. For the rest of
the section we write $\R_+:=[0,\infty)$. 

Recall that the extremal function of a compact set $K\subset \C^n$ is defined by the equation
\begin{equation*}
V_{K}(z)=\sup\{u(z)\,;\, u\in L(\C ^n),\; u|_K\leq 0 \},\qquad z\in \C^n, 
\end{equation*}
where $L(\C^n)$ is the Lelong class of all psh functions $u$ on
$\C^n$ satisfying the growth condition $u(z)-\log|z|\leq O(1)$ as
$|z|\to\infty$. The Monge-Amp\`ere measure $(dd^c V_K)^n$ is called the equilibrium measure of $K$. 

The Lelong class $L(X)$ on an $S$-parabolic manifold $X$ with a
parabolic potential $\tau$ is defined to be the class of psh functions
$u$ on $X$ satisfying the growth condition $u\leq\tau+ C$ on $X$ for
some constant $C>0$ depending  on $u$, see for example
\cite{Ayt:2011,Ayt:2014, Ayt:2016,Sna:2018,Zer:1991,Zer:1996,Zer:2000}. The
extremal function and the equilibrium measure for a compact set
$K\subset X$ is then defined analogously as in the case when $X=\C^n$,
with $L(\C^n)$ replaced by $L(X)$. In the case of the torus
$X=\C^n/\Z^n$, it is simple to show that
$\tau(z):=|\operatorname{Im}(z)|$ is a parabolic potential, so
$L(\C^n/\Z^n)$ is the class of psh functions $u$ satisfying the growth
condition $u(z)-|\operatorname{Im}(z)|\leq O(1)$ as
$|\operatorname{Im}(z)|\to \infty$.

Denote by $\omega$ the Fubini-Study metric on $\Pro^n(\C)$. If we
imbed $\C^n$ into 
$\Pro^n(\C)$ in the canonical way $(z_1,\dots,z_n)\mapsto 
[1\mathpunct{:}z_1\mathpunct{:}\cdots\mathpunct{:}z_n]$ then
$\omega$ has a psh  potential on $\C^n$, namely 
\begin{align*}
\omega= \tfrac 12 dd^c\log(1+|z|^2),\qquad z\in \C^n.
\end{align*}
We define the class $\operatorname{PSH}(\Pro^n(\C),\omega)$ 
of $\omega$-psh functions on $\Pro^n(\C )$ 
as the set of all $\varphi\in L^1( \Pro^n(\C))$, which can locally be
written as the sum of a smooth function and a psh function
and $dd^c\varphi+\omega\geq 0$.
For a compact set $K\subset \Pro^n(\C)$ we define the
$\omega$-extremal function for $K$ to be 
\begin{align*}
V_{K,\omega}(z)=\sup\{\varphi(z)\,;\, 
\varphi\in\operatorname{PSH}(\Pro^n(\C),\omega),\; \varphi|_K\leq 0  \}
\end{align*}
and the $\omega$-equilibrium measure for $K$ is $(dd^c
V_{K,\omega})^n$.  For a detailed survey on $\omega$-psh functions 
we refer to Guedj and Zeriahi \cite{Gue:2017}. See also 
Magn\'usson \cite{Mag:2011,Mag:2012}.

\subsection*{Compact subsets of $\R^n\subset \C^n$}
\label{subsec:4.1}

Although Theorem \ref{th:1.1} can be applied to calculate the extremal function for a variety of compact subsets of $\C^n$ we restrict our attention to compact subsets of $\R^n$ simply because authors in the past have shown particular interest in this case.

\begin{lemma}\label{lem:5.1}
	Let $f=(f_0,\dots,f_{n})$ be a holomorphic map on $U\subset
        \C^n$  
satisfying $f_0+\cdots+f_{n}=1$ and write 
	\begin{align*}
	K=\{z\in U\,;\, f_j(z)\in \R_+\; \text{for all }\; j \}.
	\end{align*}
	Then the function
	\begin{align*}
	\psi:&=2\log|(\sqrt{f_0},\dots,\sqrt{f_{n}})|_c\\
	&=\log\left(|f_0|+\cdots+|f_n|+\left((|f_0|+\cdots+|f_n|)^2-1\right)^{1/2} \right)			
	\end{align*}
	equals $0$ on $K$ and is maximal on $U\setminus K.$ We have
	\begin{align}
	(dd^c\psi)^n|_K=\pm\frac{n!\Omega_n}{\sqrt{f_0\cdots f_{n}}}df_1\wedge\cdots\wedge df_n.\label{jafna2}
	\end{align}
\end{lemma}

\begin{proof}
It is easy to see that $\psi=0$ on $K$. First we show that
$K=f^{-1}( \C\R^{n+1}_+)$, where 
$\C\R^{n+1}_+=\{ \alpha a\, ;\, \alpha \in
\C, a\in [0,+\infty)^{n+1}\}$. We clearly have $K\subset
f^{-1}(\C\R^{n+1}_+)$. Conversely, if $z\in f^{-1}(\C\R^{n+1}_+)$
there exist positive real numbers $x_0,\dots,x_n$ and $\theta\in
[0,2\pi]$ such that 
	\begin{align*}
	f_0(z)=x_0e^{i\theta},\; f_1(z)=x_1e^{i\theta},\;\;\dots\;\;,f_n(z)=x_ne^{i\theta}.
	\end{align*}
	Adding all the equations gives
	\begin{align*}
	1=f_0(z)+\cdots+f_n(z)=e^{i\theta}(x_0+\cdots+x_n)
\end{align*}
so $e^{i\theta}=1$ which means that $f(z)\in \R^{n+1}_+$ so
        $z\in K$. Now we apply Theorem \ref{th:1.1} with
        $\Phi=(\sqrt{f_0},\dots,\sqrt{f_n})$. We clearly have
        $\Phi^{-1}(\C\R^{n+1})=f^{-1}(\C\R_+^{n+1})=K$ and since
        $\Phi$ is a holomorphic map in a
        neighborhood of any point outside the set $S:=\{z\in
        U\,;\, f_0(z)\cdots f_n(z)=0 \}$ we see that the function
        $\psi=\log|\Phi|_c^2$ is psh and maximal on
        $U\setminus(K\cup S)$. But since $S$ is pluripolar we can
        conclude that $\psi$ is psh and maximal on $U\setminus
        K$. By simply substituting $\Phi_j$ with $\sqrt{f_j}$ in
        equation (\ref{eq:1.1}) from Theorem \ref{th:1.1} we get the
        equation 
\begin{align*}
(dd^c \psi)^n|_K
=\pm\frac{n!\Omega_n}{\sqrt{f_0\cdots
 f_{n}}}\left(\sum_{j=0}^{n}(-1)^{j+1}f_j\Lambda_{f_j}\right). 
\end{align*}
Using the assumption $f_0+\cdots+f_n=1$ it is easy to show that 
\begin{align*}
  \sum_{j=0}^{n}(-1)^{j+1}f_j\Lambda_{f_j}=df_1\wedge\cdots\wedge df_n
\end{align*}
and equation (\ref{jafna2}) follows.
\end{proof}
At this point we should point out the similarities to the Joukovski
transformation $J\colon\C\to \C$ defined by the formula
$J(z)=\frac{1}{2}(z+\frac{1}{z})$. It maps $\C\setminus
\overline{\D}(0,1)$ bijectively to $\C\setminus [-1,1]$ and its
inverse is given with the formula  
\begin{align*}
	J^{-1}(z)=z+(z^2-1)^{1/2},
\end{align*}
where the square root is chosen such that $J^{-1}(x)>1$ for $x>1$. The function $\psi$ from Lemma \ref{lem:5.1} can now be expressed as
\begin{align*}
	\psi(z)=\log J^{-1}(|f_0|+\cdots+|f_n|).
\end{align*}
We therefore see that our results, along with the extra condition
$f_0+\cdots+f_n=1$ brings us to a familiar territory, as the Joukovski
map has been extensively studied for a long time. We include three
examples which show the simplicity of this application of Theorem
\ref{th:1.1}. 
Two of them can be found in Klimek \cite{Kli:1991}, Section 5.4.
\begin{example} 
  \label{ex:5.1}
\rm \textbf{(The simplex).}
It is easy to see that
\begin{align*}
K:=\{z\in\C^n\,;\,z_j\in \R_+,\; \text{for all }\; j\; \text{ and }\; 1-z_1-\cdots-z_n\in \R_+ \}
\end{align*}
is the simplex in $\R^n\subset \C^n$. By applying Lemma \ref{lem:5.1}
we see that the extremal function for $K$ is $V_K(z)=\log J^{-1}(|f_0|+\cdots+|f_n|)$ where
\begin{align*}
f_0(z)=1-z_1-\cdots-z_n,\qquad \text{ and }\qquad f_j(z)=z_j,\;\;\;1\leq j\leq n.
\end{align*}
Using equation (\ref{jafna2}) we see that the equilibrium measure for $K$ is 
\begin{align*}
(dd^cV_K)^n|_K=\frac{ n!\Omega_n}{\sqrt{x_1\cdots x_n\cdot(1-x_1-\cdots-x_n)}}dx_1\wedge\cdots\wedge dx_n.
\end{align*}
\end{example}

\begin{example}  \rm
  \label{ex:5.3}
\textbf{(The unit ball).} Let
\begin{align*}
K=\{x\in \R^n\,;\,  1-x_1^2-\cdots-x_n^2\geq 0 \}
\end{align*}
be the unit ball in $\R^n\subset \C^n$. By Lemma \ref{lem:5.1} we see that the extremal function for $K$ is $V_K(z)=\log J^{-1}(|f_0|+\cdots+|f_n|)$ where
\begin{align*}
f_0(z)=1-z^2_1-\cdots-z^2_n,\qquad \text{ and }\qquad f_j(z)=z^2_j,\;\;\;1\leq j\leq n.
\end{align*}
Using equation (\ref{jafna2}) we see that the equilibrium measure for $K$ is 
\begin{align*}
(dd^cV_K)^n|_K=\frac{2^nn!\Omega_n}{\sqrt{1-x^2_1-\cdots-x^2_n}}dx_1\wedge\cdots\wedge dx_n.
\end{align*}
\end{example}

\begin{example}
  \label{ex:5.4}
\rm
\textbf{(The first and third quadrant of a disk).} Let
\begin{align*}
K=\{(x_1,x_2)\in \R^2\,;\,  1-x_1^2-x_2^2\geq0,\;\; x_1\cdot x_2\geq 0 \}
\end{align*}
be the union of the first and third quadrant of the unit 
disk in $\R^2\subset \C^2$. It is easy to show that
\begin{align*}
K=\{z\in \C^2\,;\, 1-z_1^2-z_2^2\in \R_+
,\;\; (z_1-z_2)^2\in \R_+,\;\; 2z_1z_2\in \R_+ \},
\end{align*}
and therefore $V_K(z)=\log J^{-1}(|f_0|+|f_1|+|f_2|)$ where
\begin{align*}
f_0(z)=1-z^2_1-z^2_2,\qquad f_1(z)=(z_1-z_2)^2\qquad 
\text{and}\qquad f_2(z)=2z_1z_2.
\end{align*}
By equation (\ref{jafna2}) we see that
\begin{align*}
(dd^cV_K)^2|_K
=\frac{ 4\sqrt{2}\pi|x_1+x_2|}
{\sqrt{x_1x_2(1-x^2_1-x^2_2)}}dx_1\wedge dx_2.
\end{align*}
\end{example}

\subsection*{Compact subsets of the complex projective space
  $\Pro^n(\C)$}
\label{subsec:4.2} We denote by $\pi:\C^{n+1}\setminus \{0\}\to
\Pro^n(\C)$ the usual projection
$\pi(z_0,\dots,z_n)=[z_0\mathpunct{:}\cdots\mathpunct{:}z_n].$  
We start with a lemma very
similar to Lemma \ref{lem:5.1}. 
\begin{lemma}\label{lem:5.5}
	Let $f_0,\dots,f_n$ be homogeneous polynomials in $\C^{n+1}$
        of degree $2k$ satisfying the equation
        $f_0+\cdots+f_n=(z_0^2+\cdots+z_n^2)^{k}$ and write  
	\begin{align*}
K=\pi\big(\{z\in \C^{n+1} \setminus\{0\} \,;\,
          (f_0(z),\dots,f_n(z))\in \R^{n+1}_+ 
          \}\big)\subset \Pro^n(\C). 
	\end{align*}
	Then the function
	\begin{align*}
\psi:=\log|(\sqrt{f_0},\dots,\sqrt{f_n})|_c^{1/k}
-\log|(z_0,\dots,z_n)| 
	\end{align*}
	equals $0$ on $K$ and $(dd^c\psi+\omega)^n=0$ on
        $\Pro^n(\C)\setminus K$. We have 
	\begin{align*}
	(dd^c\psi+\omega)^n|_K=\frac{n!\Omega_n}{2^n|z_0^2+\cdots+z_n^2|^{\frac{k(n+1)}{2}}\sqrt{f_0\cdots
          f_{n}}}\left(\sum_{j=0}^{n}(-1)^{j+1}f_j\Lambda_{f_j}\right). 
	\end{align*}
\end{lemma}

\begin{proof}
It is simple to show that
\begin{align*}
  \C\cdot\{z\in \C^{n+1}\,;\, &(f_0(z),\dots,f_n(z))\in \R^{n+1}_+ \}\\
 &=\{z\in \C^{n+1}\,;\, (f_0(z),\dots,f_n(z))\in \C\R^{n+1}_+ \}
\end{align*}
and therefore 
\begin{align*}
  K=\pi\big(\{z\in \C^{n+1}\setminus\{0\}
\,;\, (f_0(z),\dots,f_n(z))\in \C\R^{n+1}_+ \}\big).
\end{align*}
The rest of the proof is almost identical to the one of Lemma 
\ref{lem:5.1}.
\end{proof}
We consider two examples, the first one is studied in \cite{Bos:2017}.

\begin{example}
  \label{ex:5.5} \rm
\textbf{(The real projective space).} Let 
\begin{align*}
K=\Pro^n(\R)=\{[x_0\mathpunct{:}\cdots\mathpunct{:}x_n]
\,;\, x_j\in \R\;\text{ for all }\; j \}\subset \Pro^n(\C)
\end{align*}
be the real projective space. It is easy to see that
\begin{align*}
K=\pi (\{z\in \C^{n+1}\setminus\{0\}\,;\, 
z^2_j\in \R_+\; \text{ for all }\; j \})
\end{align*}
and therefore
\begin{align*}
V_{K,\omega}(z)=\log|(z_0,\dots,z_n)|_c-\log|(z_0,\dots,z_n)|, 
\qquad z=[z_0\mathpunct{:}\cdots\mathpunct{:}z_n],
\end{align*}
is the $\omega$-extremal function for $K$. We clearly have $K\cap
\C^n=\R^n$ and 
\begin{align*}
(dd^cV_{K,\omega}+\omega)^n|_{\R^n}
=\frac{n!\Omega_n}{(1+x_1^2+\cdots+x_n^2)^{\frac{(n+1)}{2}}}
dx_1\wedge\cdots\wedge dx_n.
\end{align*}
\end{example}

\begin{example}
  \label{ex:5.6} \rm
\textbf{(The first and third quadrant of the plane).} Let 
\begin{align*}
K=\pi\big(\{(x_0,x_1,x_2)\in \R^3\setminus\{0\}
\,;\, x_1\cdot x_2\geq 0\} \big)\subset \Pro^2(\C).
\end{align*}
When restricted to $\C^2$ the set $K$ becomes the union of the first 
and third quadrant of the plane $\R^2$, i.e.
\begin{align*}
K\cap \C^2=\{(x_1,x_2)\in \R^2\,;\,   x_1\cdot x_2\geq 0\}.
\end{align*}
It is relatively simple to show that $K$ can be represented as
\begin{align*}
K=\pi\big(\{(z_0,z_1,z_2)\in \C^3\setminus\{0\}\,;\, 
z_0^2\in \R_+,\; (z_1-z_2)^2\in \R_+,\; 2z_1z_2\in \R_+ \} \big)
\end{align*}
and therefore the $\omega$-extremal function for $K$ is
\begin{align*}
V_{K,\omega}(z)=
\log|(z_0,(z_1-z_2),\sqrt{2z_1z_2})|_c
-\log|(z_0,z_1,z_2)|, \quad z=[z_0,z_1,z_2].
\end{align*}
We have
\begin{align*}
(dd^cV_{K,\omega}+\omega)^2|_{K\cap \C^2}=\frac{\sqrt{2}\pi(x_1+x_2)}{(1+x_1^2+x_2^2)^{\frac{3}{2}}\sqrt{x_1x_2}}dx_1\wedge dx_2.
\end{align*}
\end{example}

\subsection*{A compact subset of the complex torus}
\label{subsec:4.3}
Finally we consider one compact subset of the complex torus $X=\C^n/ \Z ^n$. Define the functions
$f_0,\dots.,f_n$ on $X$ by 
\begin{align*}
  f_0(z)=\frac{1}{n}\sum_{j=1}^n \cos^2(\pi z_j),\qquad
  z=(z_1,\dots,z_n)\in X, 
\end{align*}
and
\begin{align*}
	f_j(z)=\frac{1}{n}\sin^2(\pi z_j),\qquad j\in \{1,\dots,n\},\; z\in X,
\end{align*}
so we clearly have $f_0+\cdots+f_n=1$ on $X$. Then the set
\begin{align*}
	K&=\{z\in X \,;\,  f_j(z)\in \R_+\; \text{for all }\; j \}\\
	&=\{z\in X\,;\,  y_j=0 \text{ or } x_j=1/2 \text{ for all $j$
          and} 
\sum_{j=1}^n\cos^2(\pi z_j)\geq 0 \}
\end{align*}
is easily seen to be compact. By Lemma \ref{lem:5.1} the function 
$$\psi=\log J^{-1}(|f_0|+\cdots+|f_n|)
$$
is maximal on $X\setminus K$. For $J\subset \{1,\dots,n \}$ write 
\begin{align*}
	K_J=K\cap \{z\in X\,;\,  y_j=0 \text{ if } j\in J \text{ and } x_j=1/2 \text{ if } j\not\in J \}.
\end{align*}
Then we have $K=\cup_J K_{J}$ and notice that $K_{\{1,\dots,n\}}=\R^n/\Z^n$ and $K_{\emptyset}=\emptyset$. To simplify notation we only calculate $(dd^c\psi)^n$ on $K_J$ when $J$ is of the type $J=\{1,2,\dots, s\}$ for some $1\leq s\leq n$ since every other $K_J$ is practically identical to some $K_J$ of this type. Indeed if $J=\{1,2,\dots, s\}$, then 
\begin{align*}
	f_0|_{K_J}=\frac{1}{n}\left(\sum_{j=0}^s \cos^2(\pi x_j) -\sum_{j=s+1}^n\sinh^2(\pi y_j)\right)
\end{align*}
and
\begin{align*}
	f_j|_{K_J}=\frac{1}{n}\sin^2(\pi x_j)\text{ if $j\leq s$ },\qquad f_j|_{K_J}=\frac{1}{n}\cosh^2(\pi y_j) \text{ if $j>s$}.
\end{align*}
By Lemma \ref{lem:5.1} we have
	\begin{align*}
(dd^c \psi)^n|_{K_J}=\frac{2^nn!\pi^n\Omega_n}{n^{n/2}\sqrt{f_0}}&|\cos(\pi x_1)\cdots \cos(\pi x_s)\sinh(\pi y_{s+1})\cdots \sinh(\pi y_n)|\\
&dx_1\wedge\cdots\wedge dx_s\wedge dy_{s+1}\wedge\cdots\wedge dy_n.
&\end{align*}

\section{Polynomial maps}
\label{sec:5}

\noindent
Let $X$ be an $n$ dimensional complex manifold and $q_0,\dots,q_{n-1}$
be holomorphic functions on $X$. For every $z\in X$ denote by
$P_z:\C\to \C$ the polynomial defined by the equation 
\begin{align*}
P_z(\zeta)&=\zeta^{n+1}-\zeta^n+q_{n-1}(z)\zeta^{n-1}-\cdots+(-1)^nq_1(z)\zeta+(-1)^{n+1}q_0(z)\\
&=\zeta^{n+1}-\zeta^n+\sum_{k=0}^{n-1}(-1)^{n+1-k}q_k(z)\zeta^k.
\end{align*}
For every $z$ denote by $f_0(z),\dots,f_n(z)$ the roots of the polynomial $P_z$ and define the set
\begin{align*}
K:&=\{z\in \C^n\,;\, \text{all the roots of $P_z$ are positive real numbers} \}\\
&=\{z\in \C^n\,;\, f_j(z)\in \R_+,\;\; 0\leq j\leq n \}.
\end{align*}
Notice that we have $f_0+\cdots+f_n=1$ so Lemma \ref{lem:5.1} applies in this situation.

\begin{theorem}\label{th:6.1}
Define $\psi\colon\C^n\to \R$ by $\psi(z)=\log J^{-1}(|f_0|+\cdots+|f_n|)$.
The function $\psi$ equals $0$ on $K$, it is maximal on 
$\C^n\setminus K$ and
\begin{align}
\big(dd^c \psi\big)^n\big|_K
=\pm\frac{n!\Omega_n}
{\sqrt{q_0(z)\Delta(z)}}dq_0\wedge\cdots\wedge dq_{n-1}.
\label{eq:6.1}
\end{align}
where $\Delta(z)$ is the discriminant of the polynomial $P_z$.
\end{theorem}

Recall that the discriminant of a polynomial $p$ of degree $n$ with
roots $s_1,\dots,s_n$ and leading coefficient $a_n$ is defined as 
\begin{align*}
\Delta:=a_n^{2n-2}\prod_{j<k}(s_j-s_k)^2
=(-1)^{\frac{n(n-1)}{2}}a_n^{2n-2}\prod_{j\not=k}(s_j-s_k).
\end{align*}
The discriminant can also be represented in terms of the derivative of
$p$, indeed we have 
\begin{align}
\prod_{j=1}^{n}p'(s_j)
=\prod_{j=1}^{n}\Bigg(a_n\prod_{\substack{k=1\\k\not=j}}^{n}(s_k-s_j)\Bigg)
=\frac{(-1)^{\frac{n(n-1)}{2}}\Delta}{a_n^{n-2}}.
\label{eq:6.2}
\end{align}
It follows easily from definition of the discriminant that for every
monic polynomial $p$ of degree $n$ with discriminant $\Delta$ and
roots $s_1,\dots,s_n$ we have 
\begin{equation}
\label{eq:6.3}
\tilde{\Delta}_j\cdot (p'(s_j))^2=\Delta,
\qquad j=1,\dots,n 
\end{equation}
where $\tilde{\Delta}_j$ is the discriminant of
$\tfrac{p(z)}{z-s_j}$. Before we prove Theorem \ref{th:6.1} we
consider a lemma. 

\begin{lemma}\label{lem:6.3}
Let $P_z$, $q_0,\dots,q_{n-1}$, $f_0,\dots,f_n$ and $\Delta$ be as in 
Theorem \ref{th:6.1}. Then we have
\begin{equation}
\label{eq:6.4}
df_1\wedge\cdots\wedge df_n
=\pm\frac{dq_0\wedge\cdots\wedge dq_{n-1}}{\sqrt{\Delta(z)}}.
\end{equation}
\end{lemma}

\begin{proof}  For every $j=0,\dots,n$ we have 
\begin{equation}\label{eq:6.5}
0=P_z(f_j(z))=f^{n+1}_j(z)-f^n_j(z)+\sum_{k=0}^{n-1}(-1)^{n+1-k}q_k(z)f^k_j(z).
\end{equation}
Differentiating equation (\ref{eq:6.5}) gives
	\begin{align}\label{wedgee}
	P'_z(f_j(z))df_j=\sum_{k=0}^{n-1}(-1)^{n-k}f^k_j(z)dq_k.
	\end{align}
	Wedging equation (\ref{wedgee}) over $1\leq j\leq n$ (notice
        we do not include $j=0$ in the wedge product) we get 
	\begin{align}
	\left(\prod _{j=1}^n P'_z(f_j(z))
          \right)df_1\wedge\cdots\wedge
          df_n=(-1)^{\frac{n(n-1)}{2}}\det(A)dq_0\wedge\cdots\wedge
          dq_{n-1}\label{poo} 
	\end{align}
	where $A$ is the matrix with coefficients
        $A_{j,k}=f_j^{k-1}(z)$.  But $A$ is a Vandermonde matrix and
        therefore we see that $(\det(A))^2$ equals the discriminant of
        the polynomial $\frac{P_z(\zeta)}{\zeta-f_0(z)}$. Now the
        result follows by multiplying $P'_z(f_0(z))$ on both sides of
        equation (\ref{poo}) and then applying equations
        (\ref{eq:6.2}) and (\ref{eq:6.3}). 
\end{proof}

\begin{proof}[Proof of Theorem \ref{th:6.1}]
The functions $f_0,\dots,f_n$ are holomorphic in a neighborhood of any
point $z\in \C^n$ such that $\Delta(z)\not=0$. Therefore, by Lemma
\ref{lem:5.1} we see that $\psi$ is maximal on 
\begin{align*}
  \C^n\setminus (K\cup \{\Delta(z)=0 \}).
\end{align*}
It is well known that the discriminant is a holomorphic function so
the set $\{\Delta(z)=0 \}$ is pluripolar and we conclude that $\psi$
is maximal on $\C^n\setminus K$. Equation (\ref{eq:6.1}) follows from
equation (\ref{jafna2}) by applying Lemma \ref{lem:6.3} and noticing that
$q_0=f_0\cdots f_n$. 
\end{proof}

Before we calculate an explicit example we prove the following lemma.

\begin{lemma}\label{lem:6.3a} The roots of the polynomial $p(z)=z^3-z^2+az-b$ are non-negative real numbers if and only if $a,b,\Delta$ are non-negative real numbers. 
\end{lemma}
\begin{proof}
	Let $s_0,s_1,s_2$ be the roots of $p$. First suppose all the roots are non-negative real numbers. Then we have
	\begin{align*}
	z^3-z^2+az-b=(z-s_0)(z-s_1)(z-s_2)
	\end{align*}
	and by comparing coefficients we see that $a,b\geq 0$. Also by definition of the discriminant we have $\Delta\geq 0$. Conversely suppose $a,b,\Delta\geq0$. Since the discriminant is non-negative all the roots $s_0,s_1,s_2$ are real numbers. We have $p(x)<0$ for all $x<0$ and therefore the roots can't be negative.
\end{proof}

\noindent
\textbf{Remark:} A similar statement for higher order polynomials is
not true. As a counter-example we can consider the polynomial 
\begin{align*}
p(z)=z^4-z^3+\tfrac{9}{4}z^2-z+\tfrac{5}{4}
\end{align*}
which has discriminant $\Delta=\frac{289}{16}$. The polynomial $p$
would satisfy the higher order analogue of Lemma \ref{lem:6.3a} but it has
$4$ complex roots, namely 
\begin{align*}
s_0=i,\;s_1=-i,\; s_2=\tfrac{1}{2}-i,\; s_3=\tfrac{1}{2}+i.
\end{align*}
\qed

\begin{example} \rm
	We apply Theorem \ref{th:6.1} with $X=\C^2$, $q_0(z)=z_2$ and $q_1(z)=z_1$ so $P_z(\zeta)=\zeta^3-\zeta^2+z_1\zeta-z_2.$
	Then, using the equation for the discriminant 
for third order polynomials, we have
	\begin{align*}
	\Delta(z)=z_1^2-4z_1^3-4z_2-27z_2^2+18z_1z_2.
	\end{align*}
	Using a program we find that
	\begin{align*}
	K&=\{(z_1,z_2)\in \C^2\,;\,  \text{all the roots of }P_z \text{ are real and positive}  \}\\
	&=\{(x_1,x_2)\in \R^2_+\,;\, \Delta(x_1,x_2)\geq 0 \}
	\end{align*}
	looks like this:
\begin{center}
			\includegraphics[width=6.5cm, height=5cm]{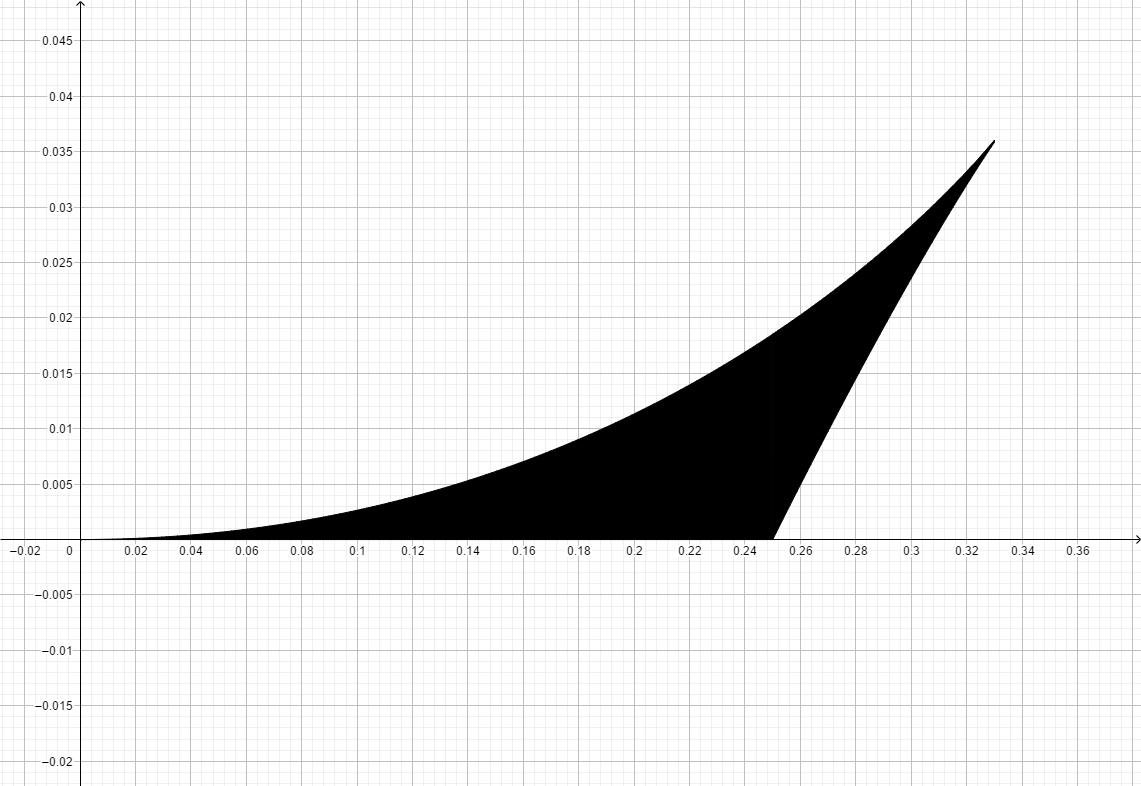}	
\end{center}
	The Monge-Ampere measure of $\psi$ defined as in Theorem \ref{th:6.1} is
	\begin{align*}
	(dd^c\psi)^2|_K=\frac{ 2\pi dx_1\wedge dx_2}{\sqrt{x_2(x_1^2-4x_1^3-4x_2-27x_2^2+18x_1x_2)}}.
	\end{align*}
	Notice that $\psi$ is not the actual extremal function for $K$ because its growth at infinity differs from the logarithm. This means that we have not calculated the actual equilibrium measure for $K$. 
\end{example}

\section{Regularization of the Lie norm}
\label{sec:6}

\noindent
In this section  we study the Levi form of the function 
$h_\varepsilon=\log v_\varepsilon$ where
$(v_\varepsilon)_{\varepsilon\geq 0}$ is the regularization of
$v=|\cdot|_c$ given for $\zeta\in \C^N$ by
\begin{equation}
  v_\varepsilon(\zeta)
=\big(|a|^2+\varepsilon|b|^2\big)^{\frac 12}
+\big(|b|^2+\varepsilon|a|^2\big)^{\frac 12}
  \label{eq:7.1}
\end{equation}
with the same notation as in the previous sections.
We see that $v_\varepsilon$ is complex homogeneous of degree $1$,
$v_\varepsilon(\zeta)\searrow v(\zeta)=|a|+|b|=|\zeta|_c$
as $\varepsilon\searrow 0$. The Levi form of $h_\varepsilon$ is quite 
involved, so in order to  simplify our calculations  
we define  the functions  $A$ and $B$ on $\C^N$ by
\begin{equation*}
 A(\zeta)=|a|^2=\tfrac{1}{2}\big(|\zeta|^2+|\scalar \zeta\zeta|\big)
\quad \text{ and } \quad
B(\zeta)=|b|^2=\tfrac{1}{2}\big(|\zeta|^2-|\scalar \zeta\zeta|\big)
\end{equation*}
and for every $\varepsilon\geq 0$,
$A_\varepsilon=A+\varepsilon B$   and
$B_\varepsilon=B+\varepsilon A$.
We observe that  
\begin{gather*}
A_{\varepsilon}+B_{\varepsilon}
=(1+\varepsilon)(A+B)=(1+\varepsilon)|\zeta|^2, \\
A_{\varepsilon}-B_{\varepsilon}=(1-\varepsilon)(A-B)
=(1-\varepsilon)|\langle \zeta,\zeta\rangle|, \quad \text{ and } \\
4A_\varepsilon B_\varepsilon=(A_\varepsilon+B_\varepsilon)^2-
(A_\varepsilon-B_\varepsilon)^2
=(1+\varepsilon)^2|\zeta|^4-(1-\varepsilon)^2|\scalar \zeta\zeta|^2.  
\end{gather*}
We define the function  $\varphi_\varepsilon$ on $\C^N$ by
\begin{align}
  \label{eq:7.4a}
  \varphi_\varepsilon(\zeta)=2
\big(A_\varepsilon(\zeta)B_\varepsilon(\zeta)\big)^{\frac 12}
&=\big((1+\varepsilon)^2|\zeta|^4
-(1-\varepsilon)^2|\scalar \zeta\zeta|^2\big)^{\frac 12}\\
&=2\big((1+\varepsilon^2)|a|^2|b|^2+\varepsilon(|a|^4+|b|^4) \big)^{\frac 12}.
\label{eq:7.4b}
\end{align}
The right hand side of (\ref{eq:7.4a})
shows that $\varphi_\varepsilon$
is a $C^\infty$ function on $\C^N\setminus\{0\}$ and
the formula for $v_\varepsilon$ becomes
\begin{equation}
v_\varepsilon(\zeta)=A_\varepsilon^{\frac 12}
+B_\varepsilon^{\frac 12}=(A_\varepsilon+B_\varepsilon
+2(A_\varepsilon B_\varepsilon)^{\frac 12}\big)^{\frac 12}
 =\big((1+\varepsilon)|\zeta|^2
+\varphi_\varepsilon(\zeta)\big)^{\frac 12}.
  \label{eq:7.5}
\end{equation}

\begin{theorem}\label{th:6a.1} 
The function $h_\varepsilon=\log v_\varepsilon$ is psh  
and maximal on $\C^N$.  
If $\zeta\in \C^N\setminus \C\R^N$ then the Levi-form 
of $h_\varepsilon$ at $\zeta$ is
\begin{align*}
\mathcal{L}_{h_\varepsilon}(\zeta,w)
=\frac{(1+\varepsilon)}{2\varphi_\varepsilon(\zeta)}
\left(\frac{4\varepsilon|\zeta|^4}{\varphi_{\varepsilon}(\zeta)^2}
|\langle w,\xi\rangle|^2+|w'|^2\right)
\end{align*}
where $w'$ denotes the component of $w$ orthogonal to $\zeta$ and $\bar{\zeta}$, and
\begin{equation}
\label{eq:3.15}
\xi=(|b|e_a-i|a|e_b)/|\zeta|
\end{equation}
is a unit vector in the plane
$\operatorname{span}\{\zeta,\bar{\zeta}\}=\operatorname{span}\{a,b\}$ 
which is perpendicular to $\zeta$, and $e_a$ and $e_b$ are
unit vectors in the direction of $a$ and $b$. In other words,
        the Levi-matrix of $h_{\varepsilon}$ has three distinct
        eigenvalues 
\begin{align*}
\lambda_0=0,\qquad 
\lambda_1=\frac{2\varepsilon(1+\varepsilon)|\zeta|^4}
{\varphi_\varepsilon(\zeta)^3},\qquad 
\lambda_2=\frac{1+\varepsilon}{2\varphi_{\varepsilon}(\zeta)},
	\end{align*}
	and the corresponding eigenspaces are
	\begin{align*}
		V_0=\operatorname{span}\{\zeta \},\qquad V_1=\operatorname{span}\{\xi \},\qquad V_2=\{\zeta,\bar{\zeta} \}^{\perp}=\{a,b \}^{\perp}.
	\end{align*}
\end{theorem}

\begin{proof}
The Levi form of $h_\varepsilon$ at $\zeta\in \C^N\setminus\{0\}$ 
is given by
\begin{equation}
  \label{eq:7.2}
  \L_{h_\varepsilon}(\zeta;w)
=\dfrac
1{v_\varepsilon(\zeta)}\bigg(
\L_{v_\varepsilon}(\zeta;w)-
\dfrac
{|\scalar{\nabla v_\varepsilon(\zeta)}{w}|^2}{v_\varepsilon(\zeta)}\bigg),
\end{equation}
where $\nabla v_\varepsilon=(\partial v_\varepsilon/\zeta_1,\dots,
\partial v_\varepsilon/\zeta_N)$.  
Since $v_\varepsilon=A_\varepsilon^{\frac 12}+B_\varepsilon^{\frac
  12}$, we have
\begin{align}
\label{eq:7.6}
\scalar{\nabla v_\varepsilon(\zeta)}w
&=\tfrac 12 A_\varepsilon^{-\frac 12}
\scalar{\nabla A_\varepsilon(\zeta)}w
+\tfrac 12 B_\varepsilon^{-\frac 12}
\scalar{\nabla B_\varepsilon(\zeta)}w,
\\
\L_{v_\varepsilon}(\zeta;w)
&=\tfrac 12 A_\varepsilon^{-\frac 12}
\L_{A_\varepsilon}(\zeta;w)
+\tfrac 12 B_\varepsilon^{-\frac 12}
\L_{B_\varepsilon}(\zeta;w)\nonumber \\
&
-\tfrac 14 A_\varepsilon^{-\frac 32}
|\scalar{\nabla A_\varepsilon(\zeta)}w|^2
-\tfrac 14 B_\varepsilon^{-\frac 32}
|\scalar{\nabla B_\varepsilon(\zeta)}w|^2.
\label{eq:7.7}
\end{align}
Assume now that $\scalar \zeta\zeta\neq 0$.  Then
\begin{equation*}
  \dfrac{\partial|\scalar\zeta\zeta|}{\partial\zeta_j}
=\dfrac{\overline{\scalar \zeta\zeta}}{|\scalar \zeta\zeta|}\zeta_j
=e^{-2i\theta}\zeta_j
\quad \text{ and } \quad 
  \dfrac{\partial^2|\scalar\zeta\zeta|}
{\partial\zeta_j\partial\bar \zeta_k}
=\dfrac{\zeta_j\bar \zeta_k}{|\scalar \zeta\zeta|},
\end{equation*}
so we get
\begin{gather*}
  \nabla A(\zeta)=\tfrac 12\big(
\bar\zeta+e^{-2i\theta}\zeta \big)=e^{-i\theta}a, \qquad
  \nabla B(\zeta)=\tfrac 12\big(
\bar\zeta-e^{-2i\theta}\zeta \big)=-ie^{-i\theta}b,\\
\L_A(\zeta;w)=\tfrac 12\big(
|w|^2+\tfrac{|\scalar \zeta w|^2}{A-B}\big),\qquad 
\L_B(\zeta;w)=\tfrac 12\big(
|w|^2-\tfrac{|\scalar \zeta w|^2}{A-B}\big),
\end{gather*}
and consequently
\begin{align}
  \label{eq:7.8}
\scalar{\nabla A_{\varepsilon}(\zeta)}w
&=e^{-i\theta}\scalar{a-i\varepsilon b}w,\\
  \label{eq:7.9}
\scalar{\nabla B_\varepsilon(\zeta)}w&=-ie^{-i\theta}\scalar{b+i\varepsilon a}w,\\
  \label{eq:7.10}
\L_{A_\varepsilon}(\zeta,w)
&=\tfrac{1}{2}\big((1+\varepsilon)|w|^2
+(1-\varepsilon)\tfrac{|\scalar \zeta w|^2}{A-B}\big),\\
  \label{eq:7.11}
\L_{B_\varepsilon}(\zeta,w)&=
\tfrac{1}{2}\big((1+\varepsilon)|w|^2
-(1-\varepsilon)\tfrac{|\scalar \zeta w|^2}{A-B}\big).
\end{align}
The function $h_\varepsilon$ is logarithmically homogeneous,
which implies $\L_{h_\varepsilon}(\zeta;\zeta)=0$ and 
that the Levi matrix of $h_\varepsilon$ has $0$ eigenvalue
with eigenvector $\zeta$. 

If  we take $w$ orthogonal to both $\zeta$ and $\bar \zeta$,
i.e., $\scalar \zeta w=\scalar{\bar \zeta}w=0$, then
$\scalar{\nabla A_\varepsilon(\zeta)}w=
\scalar{\nabla B_\varepsilon(\zeta)}w=0$, and the formulas
(\ref{eq:7.2})-(\ref{eq:7.11}) give
\begin{equation*}
  \L_{h_\varepsilon}(\zeta,w)
=\dfrac{(1+\varepsilon)}{4v_\varepsilon(\zeta)}
\bigg(\dfrac 1{A_\varepsilon(\zeta)^{\frac 12}}
+\dfrac 1{B_\varepsilon(\zeta)^{\frac 12}}
\bigg) |w|^2
=\dfrac {(1+\varepsilon)}
{2\varphi_\varepsilon(\zeta)}|w|^2.
\end{equation*}
From this formula it follows that 
$\tfrac 12 (1+\varepsilon)/\varphi_\varepsilon(\zeta)$
is an eigenvalue of the Levi matrix and that the eigenspace
contains $\{\zeta,\bar \zeta\}^\perp$, which is of dimension 
$N-1$ if $\zeta\in \C\R^N$ and of dimension $N-2$ if
$\zeta\in \C^N\setminus \C\R^N$.

Now we assume that  $\zeta\in \C^N\setminus \C\R^N$
and let $\xi=(|b|e_a-i|a|e_b)/|\zeta|$ 
be a unit vector in the span of $\zeta$ and
$\bar \zeta$ orthogonal to $\zeta$.
We know that $\xi$ is an eigenvector and the 
corresponding eigenvalue is $\L_{h_\varepsilon}(\zeta,\xi)$.  
In order to calculate $\L_{h_\varepsilon}(\zeta,\xi)$
 we first observe that by
(\ref{eq:7.8})-(\ref{eq:7.11}) we have 
\begin{align*}
  \scalar \zeta \xi
&=\dfrac{e^{i\theta}}{|\zeta|}
\scalar {a+ib}{|b|e_a-i|a|e_b}=\dfrac{2e^{i\theta}|a||b|}{|\zeta|}\\
\scalar{\nabla A_\varepsilon(\zeta)}\xi
&=\dfrac{e^{-i\theta}(1-\varepsilon)|a||b|}{|\zeta|}=-\scalar{\nabla
  B_\varepsilon(\zeta)}\xi\\
\L_{A_\varepsilon}(\zeta,\xi)&=\tfrac
12\big(1+\varepsilon+(1-\varepsilon)\dfrac{4AB}{A^2-B^2}\big)
\\
\L_{B_\varepsilon}(\zeta,\xi)&=\tfrac
12\big(1+\varepsilon-(1-\varepsilon)\dfrac{4AB}{A^2-B^2}\big).
\end{align*}
By these equations, (\ref{eq:7.6}), and (\ref{eq:7.7}) we get
\begin{gather*}
|\scalar{\nabla v_\varepsilon} \xi|^2
=\frac{AB(1-\varepsilon)^2(A_{\varepsilon}^{1/2}-B_{\varepsilon}^{1/2})^2}{4(A+B)A_{\varepsilon}B_{\varepsilon}},\\
\mathcal{L}_{v_\varepsilon}(\zeta,\xi)
=\frac{(1+\varepsilon)v_{\varepsilon}}{4A_\varepsilon^{1/2}B_{\varepsilon}^{1/2}}-\frac{(1-\varepsilon)
AB(A_\varepsilon^{1/2}-B_\varepsilon^{1/2})}{(A^2-B^2)A_{\varepsilon}^{1/2}B_{\varepsilon}^{1/2}}
-\frac{(1-\varepsilon)^2AB(A_\varepsilon^{3/2}+B_\varepsilon^{3/2})}{4(A+B)A^{3/2}_{\varepsilon}B^{3/2}_{\varepsilon}}.
\end{gather*}
By equation (\ref{eq:7.2}) we have
\begin{multline*}
\mathcal{L}_{h_\varepsilon}(\zeta,\xi)
=\frac{(1+\varepsilon)}{4A_\varepsilon^{1/2}B_{\varepsilon}^{1/2}}\\
-\frac{AB(1-\varepsilon)^2}
{4(A+B)v_{\varepsilon}A_{\varepsilon}^{1/2}B_{\varepsilon}^{1/2}}
\left(\frac{4(A_\varepsilon^{1/2}-B_\varepsilon^{1/2})}
{A_{\varepsilon}-B_{\varepsilon}}+\frac{A_\varepsilon^{3/2}+B_\varepsilon^{3/2}}
{A_\varepsilon B_\varepsilon}
+\frac{(A_\varepsilon^{1/2}-B_\varepsilon^{1/2})^2}
{A_\varepsilon^{1/2}B_\varepsilon^{1/2}v_\varepsilon} \right).
\end{multline*}
The last parenthesis equals
$(A_\varepsilon+B_\varepsilon)(A_\varepsilon^{\frac 12}
+B_\varepsilon^{\frac 12})/(A_\varepsilon B_\varepsilon)$. Hence
\begin{align*}
\mathcal{L}_{h_\varepsilon}(\zeta,\xi)
&=\frac{(1+\varepsilon)}
{4A_\varepsilon^{1/2}B_{\varepsilon}^{1/2}}-\frac{AB(1-\varepsilon)^2} 
{4(A+B)v_{\varepsilon}A_{\varepsilon}^{1/2}B_{\varepsilon}^{1/2}}\cdot \frac{(A_\varepsilon+B_{\varepsilon})v_\varepsilon}{A_{\varepsilon}B_\varepsilon}\\
&=\frac{(1+\varepsilon)}{4A_\varepsilon^{1/2}B_{\varepsilon}^{1/2}}
-\frac{AB(1-\varepsilon)^2(1+\varepsilon)}{4A^{3/2}_{\varepsilon}B^{3/2}_{\varepsilon}} \\
&=\frac{(1+\varepsilon)}{4(A_\varepsilon B_\varepsilon)^{3/2}}
\left(A_{\varepsilon}B_{\varepsilon}-AB(1-\varepsilon)^2\right)\\
&=\frac{\varepsilon(1+\varepsilon)(A+B)^2}{4(A_\varepsilon
  B_\varepsilon)^{3/2}}
=\frac{2\varepsilon(1+\varepsilon)|\zeta|^4}{\varphi_\varepsilon(\zeta)^3}.
\end{align*}
\end{proof}

\section{Proof of the main result}
\label{sec:7}

\noindent
The statement of Theorem \ref{th:1.1} is local, so 
without loss of generality we may from now on 
assume that  $X=\D^n$, where $\D$ is the unit disc in $\C$. 
In order to calculate the
Monge-Amp\`ere measure of the function
$h_\varepsilon\circ \Phi$ we need  a simple result from linear algebra.

\begin{lemma} \label{lem:8.1}
Let $D$ be an $(n+1)\times (n+1)$ Hermitian matrix and 
$A$ be an $(n+1)\times n$ matrix. Assume that $v$ is a 
unit eigenvector for $D$ with eigenvalue $0$ and denote 
by $\lambda_1,\dots,\lambda_n$ the remaining eigenvalues of $D$. 
Then
\begin{equation}\label{eq:8.1}
 \Det(A^*DA)=\lambda_1\cdots \lambda_n\cdot |\Det[A|v]|^2
\end{equation}
where $[A|v]$ is the $(n+1)\times (n+1)$ matrix 
obtained by adding $v$ as a column-vector to the 
right of the matrix $A$. 
\end{lemma}

\begin{proof} 
By a change of coordinate systems we can assume that 
$D$ is a diagonal matrix with entries $\lambda_1,\dots,\lambda_n$ and
$0$, and $v=[0,\dots,0,1]^t$. Let $\tilde A$ be the 
$n\times n$ matrix obtained by removing the bottom row of $A$ 
and $\tilde D$ be the $n\times n$ diagonal matrix with
diagonal $\lambda_1,\dots,\lambda_n$.  Then
$A^*DA=\tilde A^* \tilde D\tilde A$ and
$$
\Det(A^*DA)=\Det(\tilde A^* \tilde D\tilde A)=
\lambda_1\cdots\lambda_n|\Det(\tilde{A})|^2
=\lambda_1\cdots\lambda_n|\Det([A|v])|^2
$$
by the product formula for determinants.
\end{proof}

We have $L_{h_{\varepsilon}\circ\Phi}(z)=J_{\Phi}(z)^*L_{h_\varepsilon}(\Phi(z))
J_{\Phi}(z)$, where $L_{h_{\varepsilon}\circ\Phi}$ and
$L_{h_\varepsilon}$ denote the Levi matrices of $h_{\varepsilon}\circ
\Phi$ and $h_{\varepsilon}$ respectively. By Theorem  \ref{th:6a.1} 
$L_{h_\varepsilon}(\zeta)$ has eigenvalue $0$ with
eigenvector $\zeta$, and the others are
$2\varepsilon(1+\varepsilon)|\zeta|^4/\varphi_\varepsilon(\zeta)^3$
with multiplicity $1$ and 
$(1+\varepsilon)/2\varphi_\varepsilon(\zeta)$ with multiplicity $n-1$.
Lemma \ref{lem:8.1} gives:

\begin{theorem}\label{th:4.3}
If  $\Phi:\D^n\to \C^{n+1}\setminus\{0\}$ is a holomorphic map, then
\begin{align*}
(dd^c (h_{\varepsilon}\circ \Phi))^n=\frac{2^{n+2}n!(1+\varepsilon)^n
\varepsilon|\Phi(z)|^2|\Det[J_\Phi(z)|\Phi(z)]|^2}
{\varphi_\varepsilon(\Phi(z))^{n+2}}dV 
\end{align*}
where $J_\Phi$ is the Jacobian of $\Phi$ and $[J_\Phi|\Phi]$ is the
$(n+1)\times (n+1)$ matrix obtained by adding $\Phi$ as a column vector
to the right of $J_\Phi$. 
\end{theorem}

It is clear from the definition of $\mathcal{A}_{\Phi}$ in Theorem 
\ref{th:1.1} that
$$\mathcal{A}_{\Phi}=\{z\in \D^n \, ; \,
\Det[J_\Phi(z)|\Phi(z)]=0\}.
$$ 
The function $z\to
\Det[J_\Phi(z)|\Phi(z)]$ is holomorphic so either $\mathcal{A}_\Phi$
is pluripolar or it equals $\D^n$. If $\mathcal{A}_{\Phi}$ is pluripolar
then $\int_{\mathcal{A}_{\Phi}}(dd^c(\log|\Phi|_c))^n=0$ by
\cite[Proposition 4.6.4.]{Kli:1991}. If $\mathcal{A}_{\Phi}=\D^n$ 
then Theorem \ref{th:4.3} implies
$$(dd^c(\log|\Phi|_c))^n
=\lim_{\varepsilon\to 0}(dd^c (h_{\varepsilon}\circ\Phi))^n\equiv0.
$$
From now on we assume that 
$M= \Phi ^{-1}(\C\R^{n+1})\setminus {\mathcal A}_\Phi$ is non-empty.
In order to prove that it is a real analytic manifold of dimension 
$n$ we introduce the auxiliary map
$\tilde{\Phi}\colon\D^{n+1}\to\C^{n+1}$ defined by the equation 
\begin{align*}
\tilde{\Phi}(z,z_{n+1}):=(1+z_{n+1})\Phi(z),\qquad z\in \D^n,\; 
z_{n+1}\in \D.
\end{align*}
Notice that the Jacobians of $\tilde{\Phi}$ and $\Phi$ are related
by the equalities 
\begin{align}
J_{\tilde{\Phi}}=[(1+z_{n+1})J_\Phi|\Phi],\qquad \Det J_{\tilde{\Phi}}=(1+z_{n+1})^n\Det [J_\Phi|\Phi] \label{eq:7.2a}
\end{align}
and 
\begin{align}
\tilde{\Phi}^{-1}(\C\R^{n+1})=\Phi^{-1}(\C\R^{n+1})\times \D. \label{eq:7.3}
\end{align}
Let $z\in M$ and write $\tilde{z}=(z,0)\in \D^n\times \D$. By
        equation (\ref{eq:7.2a}) the map $\tilde{\Phi}$ is
        biholomorphic 
        in a neighborhood of $\tilde{z}$ and the variety
        $\tilde{\Phi}^{-1}(\C\R^{n+1})$ has $n+2$ real dimensions in a
        neighborhood of $\tilde{z}$. By equation (\ref{eq:7.3}) this
        means that $M$ has $n$ real dimensions in a neighborhood of
        $z$. 

Denote by 
$$
W:=\{z\in \D^{n+1} \,;\, \Det J_{\tilde \Phi}(z)\neq 0\}
=\big(\D^{n}\setminus {\mathcal A}_\Phi\big)\times \D
$$ 
the set on which $\tilde{\Phi}$ is 
locally biholomorphic and thus a
submersion. 
If $z$ is a member of $W$ then the pullback of currents by $\tilde{\Phi}$
at $z$ is well defined and we have the following result. 

\begin{theorem}\label{th:4.4} 
The $(2n+2)$-current 
\begin{align*}
	(dd^c(h\circ\Phi))^n
	\wedge \frac{i dz_{n+1}\wedge d\bar{z}_{n+1}}{2|1+z_{n+1}|^2}:=\lim_{\varepsilon\to 0}(dd^c(h_\varepsilon\circ\Phi))^n
	\wedge \frac{idz_{n+1}\wedge d\bar{z}_{n+1}}{2|1+z_{n+1}|^2}
\end{align*}
on $W\subset \D^{n+1}$ equals the pullback by
$\tilde{\Phi}\colon\D^{n+1}\to \C^{n+1}$ of the current of integration
along $\C\R^{n+1}$ of the $(n+2)$-form $-C_n|a|^{-(n+1)}d\theta\wedge
dV_{a},$ where $C_n=(-1)^{\frac{n(n-1)}{2}} n!\Omega_n$ and $\Omega_n$
is the volume of the unit ball in $\R^n$. 
\end{theorem}

\begin{proof}
To distinguish between them, we now denote by $dV_n$ and $dV_{n+1}$ 
the Euclidean volume forms on $\C^n$ and $\C^{n+1}$,
respectively. Since the result is local it is sufficient to prove it
on $\tilde{\Phi}^{-1}(L_m)$ for some $m\in \{1,\dots,n+1\}$ where
$L_m$ is defined in Proposition \ref{prop:3.2}. Indeed by 
Theorem \ref{th:4.3}, homogeneity of the function
$\varphi_{\varepsilon}$, and by equation (\ref{eq:7.2a}) we have 
\begin{align*}
(dd^c(h_\varepsilon\circ\Phi))^n
\wedge \frac{idz_{n+1}\wedge d\bar{z}_{n+1}}{2|1+z_{n+1}|^2}
&=\frac{2^{n+2}n!(1+\varepsilon)^n
\varepsilon|\Phi|^2|\Det[J_\Phi|\Phi]|^2}
{|1+z_{n+1}|^2(\varphi_\varepsilon\circ\Phi)^{n+2}}dV_{n+1}\\
&=\frac{2^{n+2}n!(1+\varepsilon)^n
\varepsilon|\tilde{\Phi}|^2
|\Det J_{\tilde{\Phi}}|^2}{(\varphi_\varepsilon
\circ\tilde{\Phi})^{n+2}}dV_{n+1}\\
&=\tilde \Phi^*\left(\frac{2^{n+2}n!(1+\varepsilon)^n 
\varepsilon|\zeta|^2}{(\varphi_\varepsilon(\zeta))^{n+2}}dV_{n+1}\right)\\
&=(-1)^{\frac{n(n-1)}{2}}n!(1+\varepsilon)^n\tilde{\Phi}^*\big(\lambda_{\varepsilon}d\theta\wedge dV_a\wedge dV_{\beta}\big),
\end{align*}
where $\lambda_{\varepsilon}$ is the function defined by the
last equality and $\theta,a,\beta$ are variables
introduced in Proposition \ref{prop:3.2}. Since the pullback
is a continuous operation under weak limits it suffices to
show that 
$$\lim_{\varepsilon\to 0}(\lambda_{\varepsilon}d\theta\wedge dV_a\wedge
        dV_{\beta})
=-\Omega_n|a|^{-(n+1)}d\theta \wedge dV_a.
$$ 
By Proposition \ref{prop:3.2} and by equation
        (\ref{eq:7.4b}) we have 
	\begin{align*}
		\lambda_{\varepsilon}(\theta,a,\beta)=\frac{-\varepsilon\big(|a|^2+|\beta|^2\big)\big(|a|-\tfrac{|\beta|^2}{|a|}\big)}{\big((1+\varepsilon^2)|a|^2|\beta|^2+\varepsilon(|a|^4+|\beta|^4) \big)^{\tfrac{n+2}{2}}},\qquad (\theta,a,\beta)\in \tilde{L}.
	\end{align*}
Now let $\chi\in C^{\infty}_0(\C^{n+1}\setminus\{0\})$ 
be a test function. We want to show that
\begin{align*}
\lim_{\varepsilon\to 0}
\int_{|\beta|<|a|}\chi(\theta,a,\beta)\lambda_{\varepsilon}(\theta,a,\beta) dV_\beta=-\chi(\theta,a,0)\Omega_n|a|^{-(n+1)}.
\end{align*}
Then the result follows by a simple application of Fubini's theorem. We
calculate this limit directly by switching into spherical coordinates
scaled by a factor of $\varepsilon^{1/2}$. We write
$r=\varepsilon^{-1/2}|\beta|$, we denote by $v$ the unit vector in the
direction of $\beta$ and by $d\sigma$ the Euclidean measure on the
$(n-1)$-dimensional unit sphere $S^{n-1}$. Hence
$\beta=\varepsilon^{1/2}rv$ and
$dV_{\beta}=\varepsilon^{n/2}r^{n-1}drd\sigma$. By the dominated
convergence theorem we have 
\begin{align*}
	&\lim_{\varepsilon\to 0}
\int_{|\beta|<|a|}\chi(\theta,a,\beta)\lambda_{\varepsilon}(\theta,a,\beta) dV_\beta\\
	=&\lim_{\varepsilon\to 0}\int_{\subalign{r&\in (0,|a|\varepsilon^{-1/2})\\ v&\in S^{n-1}}}\frac{-\varepsilon \big(|a|^2+\varepsilon r^2\big)\big(|a|-\tfrac{\varepsilon r^2}{|a|}\big)\varepsilon^{\frac{n}{2}}r^{n-1}}{(\varepsilon(1+\varepsilon^2)|a|^2r^2+\varepsilon(|a|^4+\varepsilon^2r^4))^{\tfrac{n+2}{2}}}\chi(\theta,a,\varepsilon^{1/2}rv) drd\sigma\\
	=&\int_{\subalign{r&\in(0,\infty)\\v&\in
                                              S^{n-1}}}\frac{-|a|^3r^{n-1}}{(|a|^2
                                              r^2+|a|^4)^{1+n/2}}\chi(\theta,a,0)
                                              drd\sigma\\
	=&-\operatorname{Vol}(S^{n-1})\chi(\theta,a,0)|a|^{-(n-1)}\int_{0}^{\infty}\frac{r^{n-1}}{(r^2+|a|^2)^{1+n/2}}dr\\
	=&-\operatorname{Vol}(S^{n-1})\chi(\theta,a,0)|a|^{-(n-1)}\left[\frac{r^n}{n|a|^2(|a|^2+r^2)^{n/2}} \right]_0^{\infty}\\
	=&-\Omega_n\chi(\theta,a,0)|a|^{-(n+1)}.
\end{align*}
\end{proof}
We need one more result before we prove Theorem \ref{th:1.1}. To
simplify notation we write 
$\Lambda_{\Phi_j}:=d\Phi_0\wedge\cdots\wedge
\widehat{d\Phi_j}\wedge\cdots\wedge d\Phi_n$ for $j\in \{0,\dots,n\}$ and
$V:=\tilde{\Phi}^{-1}(\C\R^{n+1})=\Phi^{-1}(\C\R^{n+1})\times
\D$. Abusing notation a little, for $z=(z',z_{n+1})\in \D^n\times \D$
we interpret $\Phi(z)$ as $\Phi(z')$, i.e.\ we use the same symbol
$\Phi$ to denote the trivial extension of $\Phi$ to $\D^{n+1}$. 

\begin{proposition}\label{prop34}
 The restrictions of the $(n+2)$-forms
         $\tilde{\Phi}^*\big(d\theta\wedge dV_a\big)$ and 
\begin{align*}
\dfrac{(1+z_{n+1})^{n+1}}
{|1+z_{n+1}|^2}
e^{-i(n+1)\theta}\Big(\sum_{j=0}^{n}(-1)^{j+1}\Phi_j\Lambda_{\Phi_j}\Big)
\wedge \tfrac i2 dz_{n+1}\wedge d\bar{z}_{n+1}
	\end{align*}
	to $V\cap W$ are equal.
\end{proposition}
\begin{proof}
	On $V$ the relation between $\Phi_j$, $a_j=a_j\circ\tilde{\Phi}$ and $\theta=\theta\circ \tilde{\Phi}$ is given by
	\begin{align*}
		(1+z_{n+1})\Phi_j=e^{i\theta}a_j,\qquad 0\leq j\leq n,
	\end{align*}
	so
	\begin{align}
		\Phi_jdz_{n+1}+(1+z_{n+1})d\Phi_j|_{V}=ie^{i\theta}a_jd\theta+e^{i\theta}da_j|_{V}.\label{jay}
	\end{align}
The Jacobian of the map $\Phi|_V\colon V\to \C\R^{n+1}$ does not
have full rank because $\Phi$ is independent of the variable
$z_{n+1}$. Therefore $d\Phi_0\wedge\cdots\wedge d\Phi_n|_{V}=0$. 
Wedging equation (\ref{jay}) over all possible $j$ gives
\begin{align}
(1+z_{n+1})^n&\Big(\sum_{j=0}^{n}(-1)^{n-j}\Phi_j\Lambda_{\Phi_j}\Big)\wedge dz_{n+1}|_V\nonumber\\
&=e^{i(n+1)\theta}\Big(dV_a+id\theta\wedge\sum_{j=0}^{n}(-1)^{j}a_j\Lambda_{a_j}\Big)|_V\label{pomm2}.
\end{align}
Let $0\leq j_0\leq n$ be any fixed number. Similarly as before we have
\begin{align}
\overline{\Phi}_{j_0}d\bar{z}_{n+1}+(1+\bar{z}_{n+1})d\overline{\Phi}_{j_0}|_V=-ie^{-i\theta}a_{j_0}d\theta+e^{-i\theta}da_{j_0}|_V.\label{pomm1}
\end{align}
Wedging equations (\ref{pomm1}) and  (\ref{pomm2}) gives
\begin{align}
\bar\Phi_{j_0}(1+z_{n+1})^n
&\Big(\sum_{j=0}^{n}(-1)^{n-j}\Phi_j\Lambda_{\Phi_j}\Big)\wedge
  dz_{n+1}\wedge d\bar z_{n+1}|_V
\nonumber\\
&=ia_{j_0}e^{in\theta}\big
(-dV_a\wedge d\theta +
(-1)^{j_0} d\theta\wedge \Lambda_{a_{j_0}}\wedge da_{j_0}\big)|_V
\nonumber
\\
&=2i(-1)^na_{j_0}e^{in\theta}d\theta\wedge dV_a|_V.
\label{pomm3}
\end{align}
After canceling $(1+\bar
z_{n+1})\overline{\Phi}_{j_0}=a_{j_0}e^{-i\theta}$ the result follows. 
\end{proof}

\begin{proof}[Proof of Theorem \ref{th:1.1}]
As we already noted after  Theorem \ref{th:4.3}, 
the Monge-Amp\`ere measure has no mass on ${\mathcal A}_\Phi$
and $M$ is an $n$ dimensional real analytic manifold,
if it is non-empty.
The map $\tilde{\Phi}$ is a
submersion on $W$.
For $\tilde{\Phi}\in \C\R^{n+1}$ we have
	\begin{align}
		|a\circ\tilde{\Phi}|e^{i\theta}=|\tilde{\Phi}|e^{i\theta}=\langle \tilde{\Phi},\tilde{\Phi}\rangle^{\tfrac{1}{2}}=(1+z_{n+1})\langle \Phi,\Phi\rangle^{\tfrac{1}{2}} \label{eq:7.6a}
	\end{align}
	where the complex square root is taken such that $\langle
        \tilde{\Phi},\tilde{\Phi}\rangle^{\tfrac{1}{2}}$ has the same
        argument $\theta$ as $\tilde{\Phi}\in \C\R^{n+1}$. 
	By combining Theorem \ref{th:4.4}, Proposition \ref{prop34} and equation (\ref{eq:7.6a}) we see that the $(2n+2)$-current
	\begin{align}
&(dd^c(h\circ \Phi))^n\wedge \frac{idz_{n+1}\wedge 
d\bar{z}_{n+1}}{2|1+z_{n+1}|^2}\label{eq:7.7a}
	\end{align}
on $W$ equals the current of integration along $V$ of the $(n+2)$-form 
	\begin{align}
	C_n\langle
          \Phi,\Phi\rangle^{-\tfrac{n+1}{2}}\Big(\sum_{j=0}^{n}(-1)^{j}\Phi_j\Lambda_{\Phi_j}\Big)\wedge
          \frac{idz_{n+1}\wedge
          d\bar{z}_{n+1}}{2|1+z_{n+1}|^2}.\label{eq:7.8a} 
	\end{align}
Let  $\omega_1:=(dd^c(h\circ \Phi))^n$ and define $\omega_2$ 
as the current of integration along $\Phi^{-1}(\C\R^{n+1})$ of 
the $n$-form $\sum_{j=0}^{n}(-1)^{j}\Phi_j\Lambda_{\Phi_j}$.
Then 
\begin{align*}
	\int_{\D^{n+1}}\chi\omega_j\wedge dz_{n+1}\wedge
  d\bar{z}_{n+1}=\Big(\int_{\D^n}\chi_1\omega_j\Big)
  \Big(\int_{\D}\chi_2 dz_{n+1}\wedge d\bar{z}_{n+1}\Big),  
\end{align*}
for $j=1,2$ and $\chi\in C^{\infty}_0(W)$ of the form 
$\chi(z)=\chi_1(z')\chi_2(z_{n+1})$.
This means that we can simply cancel out a factor of
$\tfrac{i}{2}dz_{n+1}\wedge d\bar{z}_{n+1}$ from equations
(\ref{eq:7.7a}) and (\ref{eq:7.8a}) when taking the restriction to
$W\cap\{z_{n+1}=0\}=\D^n\setminus {\mathcal A}_\Phi$.  
\end{proof}

{\small
\bibliographystyle{plain}
\bibliography{sigurdsson-snaebjarnarson-bibref}

\begin{thebibliography}{10}

\bibitem{Ayt:2011}
A.~Aytuna and A.~Sadullaev.
\newblock {$S^*$}-parabolic manifolds.
\newblock {\em TWMS J. Pure Appl. Math.}, 2(1):6--9, 2011.

\bibitem{Ayt:2014}
A.~Aytuna and A.~Sadullaev.
\newblock Parabolic {S}tein manifolds.
\newblock {\em Math. Scand.}, 114(1):86--109, 2014.

\bibitem{Ayt:2016}
A.~Aytuna and A.~Sadullaev.
\newblock Polynomials on parabolic manifolds.
\newblock In {\em Topics in several complex variables}, volume 662 of {\em
  Contemp. Math.}, pages 1--22. Amer. Math. Soc., Providence, RI, 2016.

\bibitem{Bar:1992}
M.~Baran.
\newblock Plurisubharmonic extremal functions and complex foliations for the
  complement of convex sets in {${\bf R}^n$}.
\newblock {\em Michigan Math. J.}, 39(3):395--404, 1992.

\bibitem{Bed:1976}
E.~Bedford and B.~A. Taylor.
\newblock The {D}irichlet problem for a complex {M}onge-{A}mp\`ere equation.
\newblock {\em Invent. Math.}, 37(1):1--44, 1976.

\bibitem{Bed:1982}
E.~Bedford and B.~A. Taylor.
\newblock A new capacity for plurisubharmonic functions.
\newblock {\em Acta Math.}, 149(1-2):1--40, 1982.

\bibitem{Bed:1986}
E.~Bedford and B.~A. Taylor.
\newblock The complex equilibrium measure of a symmetric convex set in {${\bf
  R}^n$}.
\newblock {\em Trans. Amer. Math. Soc.}, 294(2):705--717, 1986.

\bibitem{Blo:2009}
T.~Bloom.
\newblock Weighted polynomials and weighted pluripotential theory.
\newblock {\em Trans. Amer. Math. Soc.}, 361(4):2163--2179, 2009.

\bibitem{Bos:2017}
L.~Bos, N.~Levenberg, S.~Ma`u, and F.~Piazzon.
\newblock A weighted extremal function and equilibrium measure.
\newblock {\em Math. Scand.}, 121(2):243--262, 2017.

\bibitem{Bur:2005}
D.~Burns, N.~Levenberg, and S.~Ma'u.
\newblock Pluripotential theory for convex bodies in {$\bold R^N$}.
\newblock {\em Math. Z.}, 250(1):91--111, 2005.

\bibitem{Bur:2010}
D.~Burns, N.~Levenberg, S.~Ma'u, and Sz. R\'ev\'esz.
\newblock Monge-{A}mp\`ere measures for convex bodies and {B}ernstein-{M}arkov
  type inequalities.
\newblock {\em Trans. Amer. Math. Soc.}, 362(12):6325--6340, 2010.

\bibitem{Bur:2014}
D.~M. Burns, N.~Levenberg, and S.~Ma`u.
\newblock Extremal functions for real convex bodies.
\newblock {\em Ark. Mat.}, 53(2):203--236, 2015.

\bibitem{Dru:1974}
L.~M. Dru\.zkowski.
\newblock Effective formula for the crossnorm in complexified unitary spaces.
\newblock {\em Zeszyty Nauk. Uniw. Jagiello. Prace Mat.}, (16):47--53, 1974.

\bibitem{For:2017}
F.~Forstneri\v{c}.
\newblock {\em Stein manifolds and holomorphic mappings}.
\newblock Springer, 2nd edition, 2017.

\bibitem{Gue:2017}
V.~Guedj and A.~Zeriahi.
\newblock {\em Degenerate complex {M}onge-{A}mp\`ere equations}, volume~26 of
  {\em EMS Tracts in Mathematics}.
\newblock European Mathematical Society (EMS), Z\"urich, 2017.

\bibitem{Hor:1998}
L.~H\"ormander and R.~Sigurdsson.
\newblock Growth properties of plurisubharmonic functions related to
  {F}ourier-{L}aplace transforms.
\newblock {\em J. Geom. Anal.}, 8(2):251--311, 1998.

\bibitem{Kli:1991}
M.~Klimek.
\newblock {\em Pluripotential theory}, volume~6 of {\em London Mathematical
  Society Monographs. New Series}.
\newblock The Clarendon Press, Oxford University Press, New York, 1991.
\newblock Oxford Science Publications.

\bibitem{Lun:1985}
M.~Lundin.
\newblock The extremal {PSH} for the complement of convex, symmetric subsets of
  {${\bf R}^N$}.
\newblock {\em Michigan Math. J.}, 32(2):197--201, 1985.

\bibitem{Mag:2011}
B.S. Magn\'{u}sson.
\newblock Extremal {$\omega$}-plurisubharmonic functions as envelopes of disc
  functionals.
\newblock {\em Ark. Mat.}, 49(2):383--399, 2011.

\bibitem{Mag:2012}
B.S. Magn\'{u}sson.
\newblock Extremal {$\omega$}-plurisubharmonic functions as envelopes of disc
  functionals: generalization and applications to the local theory.
\newblock {\em Math. Scand.}, 111(2):296--319, 2012.

\bibitem{Mun:1991}
G.~Munoz, Y.~Sarantopoulos, and A.~Tonge.
\newblock Complexifications of real banach spaces, polynomials and multilinear
  maps.
\newblock {\em Studia Math.}, 134:1--33, 1999.

\bibitem{Sna:2018}
A.~S. Sn{æ}bjarnarson.
\newblock Rapid polynomial approximation on {S}tein manifolds.
\newblock {\em preprint, arXiv:1612.06173v2, {\rm (to appear in Ann.\ Pol.\
  Math.)}}, 2018.

\bibitem{Zer:1991}
A.~Zeriahi.
\newblock Fonction de {G}reen pluricomplexe \`a p\^ole \`a l'infini sur un
  espace de {S}tein parabolique et applications.
\newblock {\em Math. Scand.}, 69(1):89--126, 1991.

\bibitem{Zer:1996}
A.~Zeriahi.
\newblock Approximation polynomiale et extension holomorphe avec croissance sur
  une vari\'et\'e alg\'ebrique.
\newblock {\em Ann. Polon. Math.}, 63(1):35--50, 1996.

\bibitem{Zer:2000}
A.~Zeriahi.
\newblock A criterion of algebraicity for {L}elong classes and analytic sets.
\newblock {\em Acta Math.}, 184(1):113--143, 2000.

\end{thebibliography}

\medskip\noindent
Department of Mathematics, \\
School of Engineering and Natural Sciences, \\
University of Iceland,\\
IS-107 Reykjav\'ik, ICELAND.\\
{\tt ragnar@hi.is, audunnskuta@hi.is}
}
\end{document}